\documentclass[12pt]{amsart}

\usepackage{pinlabel}
\usepackage{amsmath,amsfonts,amssymb,latexsym,mathrsfs,stmaryrd}
\usepackage{color}
\usepackage{graphicx}
\usepackage{graphics}
\usepackage{enumerate}
\usepackage{amscd} 
\usepackage{amsxtra}
\usepackage{epic,eepic}
\usepackage{hyperref} 

\newtheorem{theorem}{Theorem}[section]
\newtheorem{prop}[theorem]{Proposition}

\newtheorem{thm}[theorem]{Theorem}
\newtheorem{lemma}[theorem]{Lemma}

\theoremstyle{definition} 
\newtheorem{defn}[theorem]{Definition}

\newcommand{\cU}{\mathcal{U}}

\newcommand{\cP}{\mathcal{P}}

\newcommand{\CC}{\mathbb{C}}

\newcommand{\qu}{/\kern-.7ex/}
\newcommand{\lqu}{\backslash \kern-.7ex \backslash}
\newcommand{\on}{\operatorname}

\newcommand{\Hom}{\on{Hom}}

\newcommand{\Pic}{\on{Pic}}

\newcommand{\NE}{\overline{\on{NE}}}

\newcommand{\bbox}{\on{Box}} 
\newcommand{\age}{\on{age}} 
\newcommand{\X}{\mathcal{X}(\boldsymbol{\Sigma})} 
\newcommand{\PX}{^{P}\mathcal{X}(\boldsymbol{\Sigma})} 
\newcommand{\QQ}{\mathbb{Q}} 
\newcommand{\CR}{\on{CR}}
\newcommand{\tw}{\on{tw}}
\newcommand{\inv}{\on{inv}}
\newcommand{\val}{\on{val}}

\begin{document}

\title[The Quantum Orbifold Cohomology of Toric Stack Bundles]{The Quantum Orbifold Cohomology of \\Toric Stack Bundles}

\author[Y. Jiang]{Yunfeng Jiang}
\address{Department of Mathematics\\ University of Kansas\\ 405 Snow Hall  1460 Jayhawk Blvd.\\ Lawrence, KS 66045, USA}
\email{y.jiang@ku.edu}

\author[H.-H. Tseng]{Hsian-Hua Tseng} 
\address{Department of Mathematics\\ Ohio State University\\ 100 Math Tower\\ 231 West 18th Ave.\\Columbus, OH 43210, USA}
\email{hhtseng@math.ohio-state.edu}

\author[F. You]{Fenglong You} 
\address{Department of Mathematics\\ Ohio State University\\ 100 Math Tower\\ 231 West 18th Ave.\\Columbus, OH 43210, USA}
\email{you.111@osu.edu}

\date{\today}

\begin{abstract} 
We study Givental's Lagrangian cone for the quantum orbifold cohomology of toric stack bundles. Using Gromov-Witten invariants of the base and combinatorics of the toric stack fibers, we construct an explicit slice of the Lagrangian cone defined by the genus $0$ Gromov-Witten theory of a toric stack bundle.
\end{abstract}
\maketitle 

\tableofcontents

\section{Introduction}
An important problem in Gromov-Witten theory is the computations of Gromov-Witten invariants of orbifolds. Genus $0$ Gromov-Witten invariants of toric stacks can be determined via a Givental-style mirror theorem proven in \cite{CCIT} and \cite{CCK}, while their higher genus Gromov-Witten invariants can be determined by Givental-Teleman reconstruction of semi-simple CohFTs \cite{Giv1}, \cite{Tel}. Toric bundles over a base $B$ are studied in \cite{SU}, where their cohomology rings were computed. Assuming knowledge about Gromov-Witten invariants of $B$, genus $0$ Gromov-Witten invariants of a toric bundle over $B$ can be determined via the mirror theorem in \cite{Brown}, while their higher genus Gromov-Witten invariants can be determined from genus $0$ invariants and localization \cite{CGT}.

Toric stack bundles, introduced by Jiang \cite{Jiang}, generalize toric bundles by using toric Deligne-Mumford stacks as fibers. The main result of this paper, Theorem \ref{main-theorem}, is a mirror theorem for toric stack bundles $\mathcal{P}\to B$. Roughly speaking, Theorem \ref{main-theorem} gives an explicit slice, the extended $I$-function, of the Lagrangian cone $\mathcal{L}_\mathcal{P}$ of the genus $0$ Gromov-Witten theory of $\mathcal{P}$, which can be used to determine all genus $0$ Gromov-Witten invariants of $\mathcal{P}$ following \cite{Giv2}, assuming that genus $0$ Gromov-Witten invariants of $B$ are known. 

Theorem \ref{main-theorem} generalizes the mirror theorems in \cite{Brown}, \cite{CCIT}. Our proof of Theorem \ref{main-theorem} follows the same approach as those in \cite{Brown}, \cite{CCIT}: localization yields a characterization result of the Lagrangian cone $\mathcal{L}_\mathcal{P}$, see Theorem \ref{characterization}. We prove that the extended $I$-function lies on $\mathcal{L}_\mathcal{P}$ by checking the conditions \textbf{(C1)}-\textbf{(C3)} in Theorem \ref{characterization}. The verification of \textbf{(C3)} for toric stack bundles involves a novel point. \textbf{(C3)} concerns fixed points of the fiberwise torus action on $\mathcal{P}$. Components of the fixed loci are abelian gerbes over the base $B$. To check \textbf{(C3)}, we need to know Gromov-Witten theory of certain abelian gerbes over $B$. Fortunately these were previously solved in \cite{AJT09}. 

The result in this paper will have applications to study birational transformation of orbifold Gromov-Witten invariants. 
An important class of crepant birational  transformation of varieties is {\em flops}. 
In the study of ordinary flops as in \cite{LLW}, the local models, which are  toric bundles over a base scheme $B$ with fibre the projective bundle over a projective space, played an important role in the proof of invariance of genus zero Gromov-Witten invariants.  A special example in our case is a toric stack bundle with fibre the weighted projective bundle over a weighted projective stack, which is the local model of {\em ordinary orbifold flop}. Theorem \ref{main-theorem} will play a crucial role to prove the crepant transformation conjecture for ordinary orbifold flops. 

The rest of the paper is organized as follows. Section \ref{sec:GW} contains a brief review of genus $0$ Gromov-Witten theory. Section \ref{sec:toric_stack} contains a review about toric stacks and related materials. The construction of toric stack bundles is recalled in Section \ref{sec:toric_bdle}. In Section \ref{sec:characterization} we apply localization to derive a characterization result of the Lagrangian cone for toric stack bundles. The main result is then proven in Section \ref{sec:proof_main_thm}. 

Throughout this paper, we work over $\mathbb{C}$. 

\subsection*{Acknowledgment} 
We thank the referees for valuable comments and suggestions. H.-H. T. thanks T. Coates, A. Corti, and H. Iritani for related collaborations. Y. J. thanks T. Coates, A. Corti and R. Thomas for the support at Imperial College London. Y. J. and H.-H. T. are supported in part by Simons Foundation Collaboration Grants. F. Y. was supported by a Presidential Fellowship of the Ohio State University during the revision of this paper. 
 
\section{Preparatory materials}

\subsection{Gromov-Witten theory}\label{sec:GW}
We give a very brief account on Gromov-Witten theory. The materials we need are discussed in more details in \cite[Section 2]{CCIT}, to which we refer the reader. 

Let $\mathcal{X}$ be a smooth proper Deligne-Mumford stack with projective\footnote{In the presence of a torus action, we may allow $\mathcal{X}$ to be only semi-projective.} coarse moduli space $X$. The {\em Chen-Ruan orbifold cohomology} $H_{\text{CR}}^*(\mathcal{X})$ of $\mathcal{X}$ is additively the cohomology of the inertia stack $\mathcal{IX}:=\mathcal{X}\times_{\mathcal{X}\times \mathcal{X}}\mathcal{X}$, where the fiber product is taken over the diagonal. The grading of $H_{\text{CR}}^*(\mathcal{X})$ is the usual grading on cohomology shifted by {\em age}. $H_{\text{CR}}^*(\mathcal{X})$ is also equipped with a non-degenerate pairing $(-,-)_{\text{CR}}$ called orbifold Poincar\'e pairing. 

Gromov-Witten invariants of $\mathcal{X}$ are defined as the following intersection numbers:
\[
\langle a_1\bar{\psi}_1^{k_1},...,a_n\bar{\psi}_n^{k_n}\rangle_{g,n,d}:=\int_{[\overline{\mathcal{M}}_{g,n}(\mathcal{X}, d)]^w}(\on{ev}_1^*a_1)\bar{\psi}_1^{k_1},...,\on{ev}_n^*(a_n)\bar{\psi}_n^{k_n},
\]
where
\begin{itemize}
\item
$\overline{\mathcal{M}}_{g,n}(\mathcal{X}, d)$ is the moduli stack of $n$-pointed genus $g$ degree $d$ stable maps to $\mathcal{X}$ with sections to all marked gerbes. 

\item
$[\overline{\mathcal{M}}_{g,n}(\mathcal{X}, d)]^w\in H_*(\overline{\mathcal{M}}_{g,n}(\mathcal{X}, d), \mathbb{Q})$ is the weighted virtual fundamental class, which is a multiple of the usual virtual fundamental class. More details can be found in \cite[Section 4.6]{AGV} and \cite[Section 2.5.1]{Tseng}.

\item
For $i=1,...,n$, $\on{ev}_i: \overline{\mathcal{M}}_{g,n}(\mathcal{X}, d)\to \mathcal{IX}$ is the evaluation map.

\item
For $i=1,...,n$, $\bar{\psi}_i\in H^2(\overline{\mathcal{M}}_{g,n}(\mathcal{X}, d), \mathbb{Q})$ are the descendant classes.

\item
$a_1,..., a_n\in H^*(\mathcal{IX})$. 
\end{itemize}
Gromov-Witten invariants can be packaged into generating functions, as follows. The genus $g$ Gromov-Witten potential of $\mathcal{X}$ is 
$$\mathcal{F}_\mathcal{X}^g(\mathbf{t}):=\sum_{n, d} \frac{Q^d}{n!}\langle \mathbf{t},...,\mathbf{t}\rangle_{g,n,d},$$
where $Q^d$ is an element in the Novikov ring of $\mathcal{X}$, $\mathbf{t}=\mathbf{t}(z)=t_0+t_1z+t_2 z^2+...\in H_{\text{CR}}^*(\mathcal{X})[z]$, and $\langle \mathbf{t},...,\mathbf{t}\rangle_{g,n,d}:=\sum_{k_1,...,k_n}\langle t_{k_1}\bar{\psi}^{k_1},...,t_{k_n}\bar{\psi}^{k_n} \rangle_{g,n,d}$.

We briefly recall the Givental's formalism about the orbifold Gromov-Witten invariants in terms of a Lagrangian cone in certain symplectic vector space, which was developed in \cite{Tseng}. Let 
\[
\mathcal{H}:=H^*_{\text{CR}}
(\mathcal {X},\mathbb{C})\otimes \mathbb{C}[[\NE(\mathcal{X})]][[z,z^{-1}]],
\]
where $\NE(\mathcal{X})$ is the Mori cone of $\mathcal{X}$. There is a $\mathbb{C}[[\NE(\mathcal{X})]]$-valued symplectic form
\[
\Omega(f,g):=Res_{z=0}(f(-z),g(z))_{\text{CR}}dz,
\]
where $(-,-)_{\text{CR}}$ is the orbifold Poincar\'e pairing. 
Let 
\[
\mathcal{H}_{+}=
H^*_{CR}(\mathcal{X},\mathbb{C})\otimes 
\mathbb{C}[[\NE(\mathcal{X})]][[z]] 
\text{ and } \mathcal{H}_{-}=
z^{-1}H^*_{CR}(\mathcal{X},\mathbb{C})\otimes 
\mathbb{C}[[\NE(\mathcal{X})]][[z^{-1}]].
\]
Then $\mathcal{H}=\mathcal{H}_{+}\oplus \mathcal{H}_{-}$ and one can think of $\mathcal{H}=T^*(\mathcal{H}_{+})$.

The graph of the differential of $\mathcal{F}_{\mathcal{X}}^0$, in the dilaton-shifted coordinates, defined a Lagrangian submanifold $\mathcal{L}_\mathcal{X}$ inside the symplectic vector space $\mathcal{H}$, more explicitly, 
\[
\mathcal{L}_{\mathcal{X}}:=
\{
(p,q)\in \mathcal{H}_{-}\oplus\mathcal{H}_{+}|p=d_{q}\mathcal{F}^{0}_{\mathcal{X}}
\}\subset\mathcal{H}.
\]
Tautological equations for genus $0$ Gromov-Witten invariants imply that $\mathcal{L}_\mathcal{X}$ is a cone ruled by a {\em finite dimensional} family of affine subspaces. A particularly important finite-dimensional slice of $\mathcal{L}_\mathcal{X}$ is the {\em $J$-function}:
$$J_\mathcal{X}(t,z):=1z+t+\sum_{n, d}\sum_\alpha \frac{Q^d}{n!}\langle t,...,t, \frac{\phi_\alpha}{z-\bar{\psi}}\rangle_{0,n+1,d}\phi^\alpha,$$
where $t=\sum_\alpha t^\alpha \phi_\alpha\in H^*_{\text{CR}}(\mathcal{X})$ and $\{\phi_\alpha\}, \{\phi^\alpha\}\subset H_{\text{CR}}^*(\mathcal{X})$ are additive bases dual to each other under $(-,-)_{\text{CR}}$. 

Another important slice of $\mathcal L_{\mathcal X}$ is given by the $S$-operator:
\begin{equation}\label{S-operator}
S_{\mathcal X}(t,z)(\gamma):= \gamma+\sum_{n, d}\sum_\alpha \frac{Q^d}{n!}\langle t,...,t,\gamma, \frac{\phi_\alpha}{z-\bar{\psi}}\rangle_{0,n+2,d}\phi^\alpha,
\end{equation}
where $\gamma\in H^*_{CR}(\mathcal X;\mathbb Q)$.

 The discussion here extends with little efforts to equivariant and twisted settings. 
 
\subsection{Preliminaries on toric stacks}\label{sec:toric_stack}
In this section we collect some basic materials concerning toric stacks. Our presentation closely follows \cite[Section 3]{CCIT}.

\subsubsection{Construction}
A toric Deligne-Mumford stack is defined by 
a stacky fan $\boldsymbol{\Sigma} =(\textbf{N},\Sigma,\rho)$, 
where 
\begin{itemize}
\item
$\textbf{N}$ is a finitely generated abelian group of rank $r$; 
\item
$\Sigma \subset \textbf{N}_{\mathbb{Q}}=\textbf{N}\otimes_{\mathbb{Z}}\mathbb{Q}$ 
is a rational simplicial fan;
\item
$\rho:\mathbb{Z}^{n}\rightarrow \textbf{N}$ is a map given by $\{\rho_1,\cdots, \rho_n\}\subset \mathbf{N}$, which is assumed to have finite cokernel. The $\rho_i$'s are vectors determining the rays of the stacky fan.
\end{itemize}
Let $\bar{\rho_{i}}$ be the image of $\rho_{i}$ under the natural map $\textbf{N} \rightarrow \textbf{N}_{\mathbb{Q}}$. 
 
The {\em fan sequence} is
\begin{equation}
0 \longrightarrow \mathbb{L}:=\text{ker}(\rho) \longrightarrow \mathbb{Z}^{n} \stackrel{\rho}{\longrightarrow} \textbf{N}.
\end{equation}

Let $\rho^{\vee}: (\mathbb{Z}^{*})^{n}\rightarrow \mathbb{L}^{\vee}$ 
be the Gale dual of $\rho$, where $\mathbb{L}^{\vee}:=H^1(\on{Cone}(\rho)^*)$ is an extension of $\mathbb{L}^{*}=\Hom(\mathbb{L},\mathbb{Z})$ 
by a torsion subgroup. More details can be found in \cite{BCS}. The  {\em divisor sequence} is
\begin{equation}
0\longrightarrow\textbf{N}^{*}\stackrel{\rho^{*}}{\longrightarrow} (\mathbb{Z}^{*})^{n}\stackrel {\rho^{\vee}}{\longrightarrow}\mathbb{L}^{\vee}.
\end{equation}

Applying $\Hom_{\mathbb{Z}}(-,\mathbb{C}^{\times})$ 
to the dual map $\rho^{\vee}$ yields a homomorphism 
\begin{equation*}
\alpha:G\rightarrow (\mathbb{C}^{\times})^{n}, \quad 
\text{where} \quad G:=\Hom_{\mathbb{Z}}(\mathbb{L}^{\vee},\mathbb{C}^{\times}),
\end{equation*}
and we let $G$ act on $\mathbb{C}^{n}$ via this homomorphism.

For $I\subset\{1,2,\cdots,n\}$, let $\sigma_I$ be the cone generated by 
$\overline{\rho}_i, i\in I$ and let $\overline{I}$ be the complement of $I$ in $\{1,2,\cdots,n\}$.  The collection of {\em anti-cones} $\mathcal{A}$ is defined as follows: 
\[
\mathcal{A}:=\left\{I\subset \{1,2,\cdots, n\}: \sigma_{\overline{I}}\in \Sigma \right\}.
\]
For $I\subset \{1,...,n\}$, define \[
\mathbb{C}^{I}=\left\{ (z_{1},\ldots,z_{n}):z_{i}=0
\text{ for } i \not\in I\right\}.
\]
 Let $\cU$ be the open subset of $\mathbb{C}^{n}$ defined as
\begin{equation*}
\mathcal{U}:=\mathbb{C}^{n}\setminus \cup_{I\not\in \mathcal{A}}\mathbb{C}^{I}. 
\end{equation*}

\begin{defn}[see \cite{BCS},  \cite{Iritani}] 
The toric Deligne-Mumford stack $\mathcal{X}(\boldsymbol{\Sigma})$ 
is defined as the quotient stack 
\[
\mathcal{X}(\boldsymbol{\Sigma}):=[\mathcal{U}/G].
\]
\end{defn}

Throughout this paper we assume the toric Deligne-Mumford stack $\mathcal{X}(\boldsymbol{\Sigma})$ has semi-projective coarse moduli space, that is, the coarse moduli space $X(\Sigma)$ is a toric variety that has at least one torus-fixed point and the natural map $X(\Sigma)\rightarrow \on{Spec}H^0(X(\Sigma),\mathcal O_{X(\Sigma)})$ is projective. See \cite[Section 3.1]{CCIT} for more details.

\begin{defn} [\cite{BCS}]
Given a stacky fan $\boldsymbol{\Sigma}=(\textbf{N},\Sigma,\rho)$, 
we define the set of box elements $\bbox(\boldsymbol{\Sigma})$ as follows
\[
\bbox(\sigma)=:
\left\{ b\in \textbf{N}: \bar{b}=
\sum\limits_{\rho_k\subseteq  \sigma}c_{k}\bar{\rho}_{k}
\text{ for some }0\leq c_{k}<1 \right\}
\]
And set $\bbox(\boldsymbol{\Sigma}):
=\cup_{\sigma\in \boldsymbol{\Sigma}}\bbox(\sigma)$
\end{defn}

The connected components of the inertia stack $\mathcal{IX}(\boldsymbol{\Sigma})$ 
are indexed by the elements of  $\bbox(\boldsymbol{\Sigma})$ (see \cite{BCS}). 
Moreover, given $b\in \bbox(\boldsymbol{\Sigma})$, the age of the corresponding connected component of 
$\mathcal{IX}$ is defined by  $\age(b):=\sum\limits_{\rho_k\subseteq \sigma} c_{k}$. 

The Picard group $\Pic(\mathcal{X}(\boldsymbol{\Sigma}))$ 
of $\mathcal{X}(\boldsymbol{\Sigma})$
can be identified with the character group 
$\Hom(G,\mathbb{C}^{\times})$. Hence
\begin{equation}
\mathbb{L}^{\vee}=
\Hom(G,\mathbb{C}^{\times})\cong 
\Pic(\mathcal{X}(\boldsymbol{\Sigma})) \cong 
H^{2}(\mathcal{X}(\boldsymbol{\Sigma});\mathbb{Z}).
\end{equation}
The inclusion $(\mathbb{C}^{\times})^{n}\subset \mathcal{U}$ 
induces an open embedding of the stack 
$\mathcal{T}=[(\mathbb{C}^{\times})^{n}/G]$ 
into $\mathcal{X}(\boldsymbol{\Sigma})$ 
and we have $\mathcal{T}\cong \mathbb{T}\times B\textbf{N}_{tor}$
with $\mathbb{T}:=(\mathbb{C}^{\times})^n/\text{Im}(\alpha)\cong 
\mathbf{N}\otimes \CC^\times$
and $\textbf{N}_{tor}\cong \ker(\alpha)$.
The Picard stack $\mathcal{T}$ acts naturally on 
$\mathcal{X}(\boldsymbol{\Sigma})$ and restricts to the $\mathbb{T}$-action on $\mathcal{X}(\boldsymbol{\Sigma})$. 
A $\mathcal{T}$-equivariant line bundle on 
$\mathcal{X}(\boldsymbol{\Sigma})$ 
corresponds to a $(\mathbb{C}^{\times})^{n}$-equivariant 
line bundle on $\mathcal{U}$. Thus, 
\[
\Pic_{\mathcal{T}}(\mathcal{X}(\boldsymbol{\Sigma}))\cong \Hom((\mathbb{C}^{\times})^{n},\mathbb{C}^{\times})\cong (\mathbb{Z}^{n})^{*}.
\]
We write $u_{1},\ldots,u_{n}$ for the basis of 
$\mathcal{T}$-equivariant line bundles on 
$\mathcal{X}(\boldsymbol{\Sigma})$ 
corresponding to the standard basis of $(\mathbb{Z}^{n})^{*}$ 
and write $D_{1},\ldots,D_{n}$ 
for the corresponding non-equivariant line bundles, 
i.e. $$D_{i}=\rho^{\vee}(u_{i}).$$
By abuse of notation, 
we also write $u_{i}$ and $D_{i}$ for the corresponding first Chern classes.

\subsubsection{$S$-extended stacky fan}
Given a stacky fan $\boldsymbol{\Sigma}=(\textbf{N},\Sigma,\rho)$ 
and a finite set 
\begin{equation*}
S=
\{s_{1},\ldots,s_{m}\}\subset \textbf{N}. 
\end{equation*}
The $S$-extended stacky fan in the sense of \cite{Jiang} is given by 
$(\textbf{N},\Sigma, \rho^{S} )$, 
where 
\begin{equation}
\rho^{S}: \mathbb{Z}^{n+m}\rightarrow \textbf{N}, \quad \rho^{S}(e_{i}) := \left\{
     \begin{array}{lr}
      \rho_{i} &  1\leq i \leq n;\\
       s_{i-n} &  n+1\leq i \leq n+m.
     \end{array}
   \right.
\end{equation}
Let $\mathbb{L}^{S}$ be the kernel of 
$\rho^{S}: \mathbb{Z}^{n+m}\rightarrow \textbf{N}$. 
Gale duality of the $S$-extended fan sequence
\begin{equation}\label{S-ext-fan-seq}
0\longrightarrow \mathbb{L}^{S}:=\text{ker}(\rho^S) \longrightarrow \mathbb{Z}^{n+m} \stackrel{\rho^{S}}{\longrightarrow} \textbf{N}
\end{equation}
yields the $S$-extended divisor sequence
\begin{equation}\label{S-extended-divisor-seq}
0\longrightarrow\textbf{N}^{*}\stackrel{\rho^{*}}{\longrightarrow} (\mathbb{Z}^{*})^{n+m}\stackrel {\rho^{S \vee}}{\longrightarrow}(\mathbb{L}^{S })^{\vee},
\end{equation}
where $(\mathbb{L}^{S})^{ \vee}$ is the Gale dual of $\rho^S$.
As in \cite[Section 4]{CCIT}, $(\mathbb{L}^{S})^{\vee}$ 
is the $S$-extended Picard group of $\mathcal{X}(\boldsymbol{\Sigma})$.

Let $\mathcal{A}^{S}$ be the collection of $S$-extended anti-cones, i.e. 
\begin{equation*}
\mathcal{A}^{S}:=\left\{I^{S}\subset \{1,2,\cdots,n+m\}: \sigma_{\overline{I}^S}\in \Sigma\right\}. 
\end{equation*}
Note that 
\[
\{s_{1},\ldots,s_{m}\}\subset I^{S}, \quad \forall I^{S}\in \mathcal{A}^{S}.
\]

By applying $\Hom_{\mathbb{Z}}(-,\mathbb{C}^{\times})$ 
to the $S$-extended dual map $\rho^{\vee}$, we have a homomorphism 
\begin{equation*}
\alpha^{S}:G^{S} \rightarrow 
(\mathbb{C}^{\times})^{n+m}, 
\quad \text{where}\quad G^{S}
:=\Hom_{\mathbb{Z}}((\mathbb{L}^{S})^{\vee}, \mathbb{C}^{\times}).  
\end{equation*}
Define $\mathcal{U}^S$ to be the open subset of $\mathbb{C}^{n+m}$ defined by $\mathcal{A}^{S}$:
\begin{equation*}
\mathcal{U}^{S}:=
\mathbb{C}^{n+m}\setminus \cup_{I^{S}
\not\in \mathcal{A}^{S}}\mathbb{C}^{I^{S}}
=\mathcal{U}\times (\mathbb{C}^{\times})^{m},
\end{equation*}
where 
\begin{equation*}
\mathbb{C}^{I^{S}}=
\left\{(z_{1},\ldots, z_{n+m}):z_{i}=
0\text{ for }i\not \in I^{S}\right\}. 
\end{equation*}
Let $G^{S}$ act on $\mathcal{U}^{S}$ via $\alpha^{S}$. 
Then we obtain the quotient stack $[\mathcal{U}^{S}/G^{S}]$. 
Jiang \cite{Jiang} showed that 
\begin{equation*}
[\mathcal{U}^{S}/G^{S}]\cong [\mathcal{U}/G]=\mathcal{X}(\boldsymbol{\Sigma}).
\end{equation*}

\subsubsection{Toric maps from $\mathbb{P}_{r_{1},r_{2}}$ to $\mathcal{X}(\boldsymbol{\Sigma})$}
We recall the discussion in \cite[Section 3.5]{CCIT} on maps from $1$-dimensional toric stacks to a toric stack.  For positive integers $r_{1}$ and $r_{2}$ let $\mathbb{P}_{r_{1},r_{2}}$ be the unique toric Deligne-Mumford stack such that 
\begin{itemize}
\item
its coarse moduli space is $\mathbb{P}^1$;
\item
its isotropy group  at $0\in \mathbb{P}^{1}$ is $\mu_{r_{1}}$;
\item
its isotropy group  at $\infty \in\mathbb{P}^{1}$ is $\mu_{r_{2}}$;
and 
\item
there are no non-trivial orbifold structures at other points.
\end{itemize}

As in \cite{BCS}, For an extended stacky fan $\boldsymbol{\Sigma}$, 
let $\sigma\in \Sigma$ be a cone, define 
\[
\mbox{link}(\sigma):=\{\tau: \sigma+\tau\in\Sigma,\sigma\cap\tau=0\},
\]
and $\rho_{i_{1}},\ldots,\rho_{i_{l}}$ be the rays in $\mbox{link}(\sigma)$. 
A cone $\sigma\in \Sigma$ defines a closed substack of $\mathcal{X}(\mathbf{\Sigma})$, which is the toric stack $\mathcal{X}(\boldsymbol{\Sigma}/\sigma)$ corresponding to the quotient stacky fan $(\mathbf{N}(\sigma), \Sigma/\sigma, \rho(\sigma))$, where $\Sigma/\sigma$ is the quotient fan in
$\textbf{N}(\sigma)_\QQ=(\textbf{N}/\sum_{i\in\sigma} \mathbb Z \rho_i)\otimes\QQ$.
More precisely, 
$\boldsymbol{\Sigma}/\sigma=
(\textbf{N}(\sigma),\Sigma/\sigma,\rho(\sigma))$ 
is an extended stacky fan, 
where $\rho(\sigma):\mathbb{Z}^{l+m}\rightarrow 
\textbf{N}(\sigma)$ 
is given by the images of $\rho_{i_{1}},\ldots,\rho_{i_{l}}$, 
$s_{1}, \ldots, s_{m}$ under 
$\textbf{N}\rightarrow 
\textbf{N}(\sigma)$. 
From the construction of extended toric Deligne-Mumford stack, we have 
\[
\mathcal{X}(\boldsymbol{\Sigma}/\sigma):=
[\mathcal{U}^{S}(\sigma)/G^{S}(\sigma)]
\]
where 
$\mathcal{U}^{S}(\sigma)=
(\mathbb{C}^{l}-V(J_{\Sigma/\sigma}))\times 
(\mathbb{C}^{\times})^{m}=
\mathcal{U}(\sigma)\times 
(\mathbb{C}^{\times})^{m}$, 
$G^{S}(\sigma)=
Hom_{\mathbb{Z}}(\mathbb{L}^{S\vee}(\sigma),\mathbb{C}^{\times})$. 
 
For a box element $b\in \bbox(\boldsymbol{\Sigma})$, 
let $\mathcal{X}(\boldsymbol{\Sigma})_{b}$ 
be the component of the inertia stack $\mathcal{IX}(\boldsymbol{\Sigma})$ 
corresponding to $b$. 
Then $\mathcal{X}(\boldsymbol{\Sigma})_{b}\cong \mathcal{X}(\boldsymbol{\Sigma}/\sigma(b))$, 
where $\sigma(b)$ is the minimal cone containing 
$\bar{b}$. We define $b_{i}\in[0,1), 1\leq i \leq n$ 
by the condition $\bar{b}=\sum^{n}_{i=1}b_{i}\bar{\rho}_{i}$, 
note that $b_{i}=0$ for $\overline{\rho}_i\not\in \sigma(b)$.

\begin{defn}[see \cite{CCIT}, Notation 8]
Let  $\sigma,\sigma^{\prime}\in \Sigma$ be two top dimensional cones, 
we write $\sigma \dagger \sigma^\prime$ 
if they intersect along a codimension-1 face and 
we denote $j$ to be the unique index such that 
$\bar{\rho}_{j}\in\sigma\setminus \sigma^{\prime}$, 
and $j^{\prime}$ to be the unique index such that 
$\bar{\rho}_{j^{\prime}}\in\sigma^{\prime}\setminus \sigma$. 
\end{defn}

\begin{prop}[\cite{CCIT}, Proposition 10]\label{toric-morphism}
Let $\mathcal{X}(\boldsymbol{\Sigma})$ be 
the toric Deligne-Mumford stack associated to 
a stacky fan $\boldsymbol{\Sigma} =(\textbf{N},\Sigma,\rho)$. 
Suppose top dimensional cones $\sigma, \sigma^{\prime}$ satisfy 
$\sigma \dagger \sigma^\prime$ and $b\in Box(\sigma)$. 
The following are equivalent:
\begin{itemize}
\item A positive rational number $c$ such that 
$\langle c\rangle =\hat{b}_{j}$, where $\hat{b}=\inv(b)$ is the involution of $b$.
\item A representable toric morphism 
$f: \mathbb{P}_{r_{1},r_{2}}\rightarrow 
\mathcal{X}(\boldsymbol{\Sigma})$ 
such that $f(0)=\mathcal{X}(\boldsymbol{\Sigma/\sigma})$, 
$f(\infty)=\mathcal{X}(\boldsymbol{\Sigma/\sigma^{\prime}})$ 
and the restriction 
$f|_{0}:B\mu_{r_{1}}\rightarrow \mathcal{X}(\boldsymbol{\Sigma/\sigma})$ 
gives the box element $\hat{b}\in \bbox(\sigma)$.

\end{itemize}
\end{prop}
The data $\sigma, \sigma^{\prime}, b$ and 
$c$ determine the map 
$f: \mathbb{P}_{r_{1},r_{2}}\rightarrow 
\mathcal{X}(\boldsymbol{\Sigma})$ and 
determine the rational number $r_{2}$ and 
the box element $b^{\prime}\in \bbox(\sigma^{\prime})$ 
given by the restriction 
$f|_{\infty}:B\mu_{r_{2}}\rightarrow 
\mathcal{X}(\boldsymbol{\Sigma})$. 
More precisely, 
$b^{\prime}$ is the unique element of $\bbox(\sigma^{\prime})$ such that 
\begin{equation}\label{b^prime}
\hat{b}+\lfloor c\rfloor \rho_{j}+q^{\prime}\rho_{j^{\prime}}+b^{\prime}
\equiv 0 \quad \text{mod} \bigoplus
\limits_{i\in \sigma\cap \sigma^{\prime}}\mathbb{Z}\rho_{i}
\end{equation}
for some $q^{\prime}\in \mathbb{Z}_{\geq 0}$. 
As in \cite[Definition 12]{CCIT}, 
define $d_{c,\sigma,j}$ to be the element of 
$\mathbb{L}\otimes\mathbb{Q}$ 
satisfying the relation
\[
c\bar{\rho}_{j}+
\left(
\sum\limits_{i\in \sigma\cap\sigma^{\prime}}c_{i}\bar{\rho}_{i}
\right)
+c^{\prime}\bar{\rho}_{j^{\prime}}=0
\]
such that 
\[
D_{j}\cdot d_{c,\sigma,j}=c, 
\quad D_{j^{\prime}}\cdot d_{c,\sigma,j}=c^{\prime}, 
\quad D_{i}\cdot d_{c,\sigma,j}=c_{i} \text{ for } i\in \sigma\cap \sigma^{\prime},
\]
and
\[
D_{i}\cdot d_{c,\sigma,j}=0 \text{ for }i\not\in \sigma\cup\sigma^{\prime}.
\]
Hence, $d_{c,\sigma,j}$ is the degree of 
the representable toric morphism 
$f:\mathbb{P}_{r_{1},r_{2}}\rightarrow  
\mathcal{X}(\boldsymbol{\Sigma})$. 
Let $\Lambda E^{\sigma^{\prime},b^{\prime}}_{\sigma,b}\subset 
\mathbb{L}\otimes \mathbb{Q}$ 
to be the set of degrees $d_{c,\sigma,j}$ 
representable toric morphisms 
$f:\mathbb{P}_{r_{1},r_{2}}\rightarrow  
\mathcal{X}(\boldsymbol{\Sigma})$ such that 
$f(0)=\mathcal{X}(\boldsymbol{\Sigma/\sigma})$, 
$f(\infty)=\mathcal{X}(\boldsymbol{\Sigma/\sigma^{\prime}})$ 
and $f|_{0}$ and 
$f|_{\infty}$ give the box elements $\hat{b}$ and $b^{\prime}$, respectively. 
More precisely,
\[
\Lambda E^{\sigma^{\prime},b^{\prime}}_{\sigma,b}=
\left\{
d_{c,\sigma,j}\in \mathbb{L}\otimes \mathbb{Q}:
c>0 \text{ such that } \langle c \rangle =\hat{b}_{j} \text{ and } 
b^{\prime} \text{ satisfies } (\ref{b^prime})
\right \},
\]
see \cite[Definition 14]{CCIT}.

We recall a few notions related to extended degrees for toric stacks.
\begin{defn}[\cite{CCIT}, Definition 22]
Consider a cone $\sigma\in \Sigma$, 
let $\Lambda^{S}_{\sigma}\subset\mathbb{L}^{S}\otimes\mathbb{Q}\subset \mathbb{Q}^{n+m}$ 
be the set of elements 
$\lambda=\sum\limits_{i=1}^{n+m}\lambda_{i}e_{i}$ 
such that 
\[
\lambda_{n+j}\in \mathbb{Z}, \quad 1\leq j\leq m; 
\quad \lambda_{i} \in \mathbb{Z}, \text{ if }i\not\in \sigma \text{ and }1\leq i\leq n.
\]
Set $\Lambda^{S}:=\cup_{\sigma\in \Sigma}\Lambda^{S}_{\sigma}$.
\end{defn}

\begin{defn}[\cite{CCIT}, Definition 23]
The reduction function $v^{S}$ is defined by
\begin{align*}
v^{S}:\Lambda^{S}& 
\longrightarrow \bbox(\boldsymbol{\Sigma})\\
\lambda & 
\longmapsto \sum\limits_{i=1}^{n}\lceil \lambda_{i}\rceil \rho_{i}+
\sum\limits_{j=1}^{m}\lceil \lambda_{n+j}\rceil s_{j} 
\end{align*} 
Hence, we have $\overline{v^{S}(\lambda)}=
\sum^{n}_{i=1}\langle -\lambda_{i}\rangle \bar{\rho}_{i}\in\sigma$ 
for $\lambda\in \Lambda^{S}_{\sigma}$.
We introduce the following sets:
\[
\Lambda^{S}_{b}:=\{\lambda\in \Lambda^{S}:v^{S}(\lambda)=b\}
\]
\[
\Lambda E^{S}:= \Lambda^{S}\cap \NE^{S}(\mathcal{X}(\boldsymbol{\Sigma}))
\]
\[
\Lambda E^{S}_{b}:= \Lambda^{S}_{b}\cap \NE^{S}(\mathcal{X}(\boldsymbol{\Sigma}))
\]
\end{defn}

\section{Toric stack bundles}\label{sec:toric_bdle}
\subsection{Construction}
Let $P\rightarrow B$ be a principal $(\mathbb{C}^{\times})^{n+m}$-bundle 
over a smooth projective variety $B$, 
we introduce the toric stack bundle 
$\PX$. 
\begin{defn}[\cite{Jiang}]
The toric stack bundle 
$\pi: \mathcal{P}:=~
\PX\rightarrow B$ is defined 
to be the quotient stack
\[
^{P}\mathcal{X}(\boldsymbol{\Sigma}):=
[(P\times _{(\mathbb{C}^{\times})^{n+m}}\mathcal{U}^{S})/G^{S}]
\]
where $G^{S}$ acts on $P$ trivially.
\end{defn}
It is shown in \cite{Jiang} that $\mathcal{P}$ is a smooth Deligne-Mumford stack.

We now recall the description of the inertia stack of $\mathcal{P}$. We have an action of $(\mathbb{C}^{\times})^{n+m}$ on 
$\mathcal{U}^{S}(\sigma)$ 
induced by the natural action of 
$(\mathbb{C}^{\times})^{l+m}$ on 
$\mathcal{U}^{S}(\sigma)$ and 
the projection 
$(\mathbb{C}^{\times})^{n+m}\rightarrow 
(\mathbb{C}^{\times})^{l+m}$.
We let
\begin{align*}
^{P}\mathcal{X}(\boldsymbol{\Sigma}/\sigma)& =[(P\times_{(\mathbb{C}^{\times})^{n+m}}(\mathbb{C}^{\times})^{l+m}\times _{(\mathbb{C}^{\times})^{l+m}}\mathcal{U}^{S}(\sigma))/G^{S}(\sigma)]\\
& =[(P\times_{(\mathbb{C}^{\times})^{n+m}}\mathcal{U}^{S}(\sigma))/G^{S}(\sigma)]
\end{align*}
be the quotient stack. 
By \cite[Proposition 3.5]{Jiang}, $^{P}\mathcal{X}(\boldsymbol{\Sigma}/\sigma)$ is a closed substack of $\mathcal{P}$.

\begin{prop}[\cite{Jiang}, Proposition 3.6]
Let $\pi: \mathcal{P}\rightarrow B$ be a toric stack bundle over a smooth variety $B$ with fibre the toric Deligne-Mumford stack $\X$ associated to the extended stacky fan $\boldsymbol{\Sigma}$, then the inertia stack of $\mathcal{P}$ is
\[
\mathcal{IP}=
\coprod\limits_{b\in \bbox(\boldsymbol{\Sigma})}\mathcal{P}_{b}:=
\coprod\limits_{b\in \bbox(\boldsymbol{\Sigma})}{^{P}\mathcal{X}(\boldsymbol{\Sigma}/\sigma(\bar{b}))}.
\] 
The age of $\mathcal{P}_{b}$ is the same as the age of $\X_{b}$.
\end{prop}
For the principal $(\mathbb{C}^{\times})^{n+m}$-bundle 
$P=\oplus^{n+m}_{j=1}L^{*}_{j}$ over $B$, 
where $L_{j}$ is the corresponding $j$-th line bundle. Let 
\begin{equation}
U_{j} = \left\{
     \begin{array}{lr}
      u_{j}-c_{1}(L_{j}) &  1\leq j \leq n;\\
       0 &  n+1\leq j \leq n+m.
     \end{array}
   \right.
\end{equation}
By abuse of notation, we also denote $U_j$ for 
the corresponding $\mathbb T$-equivariant line bundle over $\mathcal P$.

\subsection{Main result}
We choose an integral basis $\{p_1,\ldots,p_{n-r}\}$ of $\mathbb L^{\vee}$. The toric stack bundle $\mathcal{P}$ is endowed with $n-r$ tautological line bundles whose first Chern classes we denote by $-P_{1},\ldots,-P_{n-r}$. They restrict to the corresponding first Chern classes $-p_{1},\ldots,-p_{n-r}$ on the fiber. 
Recall that the $\mathbb T$-equivariant Novikov ring of the toric stack $\mathcal X (\boldsymbol{\Sigma})$ is defined as
\[
\Lambda_{nov}^{\mathbb T}:=S_{\mathbb T}[[\NE(\mathcal{X}(\boldsymbol{\Sigma}))\cap H_{2}(\X;\mathbb{Z})]],
\]
where $S_{\mathbb T}$ is the fraction field of $R_{\mathbb T}:=H^*_{\mathbb T}(pt,\mathbb C)$ and $\NE(\mathcal{X}(\boldsymbol{\Sigma}))$ is the Mori cone of $\mathcal{X}(\boldsymbol{\Sigma})$.

For $\mathcal{D}\in H_{2}(\mathcal{P})$, let $\mathfrak D:=\pi_{*}(\mathcal{D})\in H_{2}(B)$ be its projection to the base and let
\[
\lambda=(d,k)\in\mathbb{L}^{S}\otimes\mathbb{Q},
\]
under the canonical splitting 
$\mathbb{L}^{S}\otimes
\mathbb{Q}\cong 
(\mathbb{L}\otimes
\mathbb{Q})\oplus\mathbb{Q}^m $ be the fiber class, 
such that 
$\langle P_{i},\mathcal{D}\rangle=
\langle p_{i},d\rangle$. 
Hence $\mathcal D$ is represented by $Q^{\mathfrak D} q^d$ in the Novikov ring of $\mathcal P$.

Let $J_{B}(z,\tau)= \sum\limits_{\mathfrak D\in \NE(B)}J_{\mathfrak D}(z,\tau)Q^{\mathfrak D}$  be the decomposition of the $J$ function of $B$ according to the degree of curves.

\begin{defn}\label{defn:I_func}
We introduce the hypergeometric modification 
(The $S$-extended $\mathbb{T}$-equivariant $I$-function of the toric stack bundle $\mathcal{P}$)
\begin{align*}
& I^{S}_{\mathcal{P}}(z,t,\tau,q,x,Q):=\\
& e^{\sum^{n}_{i=1}U_{i}t_{i}/z}
\sum\limits_{\mathfrak D\in \NE(B)}
\sum\limits_{b\in \bbox(\boldsymbol{\Sigma})}
\sum\limits_{\lambda\in \Lambda E^{S}_{b}}
J_{\mathfrak D}(z,\tau)Q^{\mathfrak D}\tilde{q}^{\lambda}
e^{\lambda t}
\left( 
\prod\limits^{n+m}_{i=1}
\frac
{\prod_{\langle a\rangle =\langle \lambda_{i}-c_{1}(L_{i})\cdot\mathfrak D\rangle,a\leq 0}(U_{i}+az)}
{\prod_{\langle a\rangle =\langle \lambda_{i}-c_{1}(L_{i})\cdot\mathfrak D\rangle,a\leq \lambda_{i}-c_{1}(L_{i})\cdot\mathfrak D}(U_{i}+az)}
\right)
\textbf{1}_{b}
\end{align*}
where 
\begin{enumerate}
\item for each $\lambda\in \Lambda E^{S}_{b}$, 
we write $\lambda_{i}$ for the $i$th component of 
$\lambda$ as an element of $\mathbb{Q}^{n+m}$. 
We have $\langle \lambda_{i}\rangle=\hat{b}_{i}$ 
for $1\leq i \leq n$ and 
$\langle \lambda_{i}\rangle =0$ 
for $n+1\leq i\leq n+m$.\\
\item $U_{i}:=0$, if $n+1\leq i\leq n+m$.\\
\item $\textbf{1}_{b}$ is 
the identity class supported on the twisted sector 
$\X_{b}$ 
associated to $b\in \bbox(\boldsymbol{\Sigma})$;\\
\item $t=(t_{1},\ldots,t_{n})$ are variables, 
and $e^{\lambda t}:=
\prod^{n}_{i=1}e^{\langle D_{i},d\rangle t_{i}}$
\item for $\lambda=(d,k) \in \Lambda E^{S}\subset \mathbb{L}^{S}\otimes \mathbb{Q}$, 
we have $k\in (\mathbb{Z}_{\geq 0})^m$ and 
$d\in \NE(\mathcal{X}(\boldsymbol{\Sigma}))\cap H_{2}(\X;\mathbb{Z})$, 
we write 
$\tilde{q}^{\lambda}=
q^{d}x^{k}=
q^{d}x_{1}^{k_{1}}\cdots x_{m}^{k_{m}}
\in \Lambda^{\mathbb{T}}_{nov}[[x]]$, 
with variables $x=(x_{1},\ldots, x_{m})$.
\end{enumerate}
\end{defn}

\begin{defn}[\cite{CCIT}]
A $\Lambda^{\mathbb{T}}_{nov}[[x]]$-valued point of $\mathcal L_{\mathcal P}$ is an element of $\mathcal H[[x]]$ of the form
\[
-1z+\textbf{t}(z)+
\sum\limits^{\infty}_{n=0}
\sum\limits_{\substack{d\in \NE(\X)\\ \mathfrak D\in \NE(B)}}
\sum\limits_{\alpha}
\frac{Q^{\mathfrak D}q^{d}}{n!}
\langle 
\textbf{t}(\bar{\psi}),\ldots,\textbf{t}(\bar{\psi}),
\frac{\phi_{\alpha}}{-z-\bar{\psi}}
\rangle ^{\mathbb{T}}_{0,n+1,\mathcal{D}}\phi^{\alpha}
\]
for some $\textbf{t}(z)\in \mathcal{H}_{+}[[x]]$ 
with $\textbf{t}|_{Q=q=x=0}=0$. 

\end{defn}

The following is the main result of this paper.
\begin{theorem}\label{main-theorem}
The hypergeometric modification
$I^{S}_{\mathcal{P}}(z,t,\tau,q,x,Q)$ is 
a $\Lambda^{\mathbb{T}}_{nov}[[x,t]]$-valued point
of the Lagrangian cone 
$\mathcal{L}_{\mathcal{P}}$ 
for the $\mathbb{T}$-equivariant Gromov-Witten theory 
of $\cP$.
\end{theorem}
The rest of this paper is devoted to a proof of Theorem \ref{main-theorem}.

\section{Localization methods in toric Gromov-Witten theory}\label{sec:characterization}
In this Section we describe a characterization of the Lagrangian cone of a toric stack bundle $\mathcal{P}$ via localization. 

\subsection{Lagrangian cones for toric stack bundles}
Given a toric Deligne-Mumford stack $\mathcal{X}(\boldsymbol{\Sigma})$ associated to an extended stacky fan $\boldsymbol{\Sigma}$. 
The maximal torus $\mathbb{T}$ acts on the toric Deligne-Mumford stack $\mathcal{X}(\boldsymbol{\Sigma})$,
hence acts on the toric stack bundle $\mathcal{P}= ~^{P}\mathcal{X}(\boldsymbol{\Sigma})$.
The fixed points under the torus action correspond to the top dimensional cones in the fan $\Sigma$.
A top dimensional cone $\sigma$ gives a fixed point section\footnote{We abuse notation here: $\mathcal{P}_{\sigma}$ are gerbes over $B$ which may not have sections.} 
 $\mathcal{P}_{\sigma}:=~^{P}\mathcal{X}(\boldsymbol{\Sigma}/\sigma)$ for the toric stack bundle $\mathcal P$.
Note that $\mathcal P_\sigma$ is an abelian gerbe over the base $B$: it is a fiber product of root gerbes associated to the line bundles defining $\mathcal{P}$. We write $N_{\sigma}\mathcal{P}$ for the normal bundle at the $\mathbb{T}$-fixed section $\mathcal{P}_{\sigma}$. 

For the rest of this paper, we write $\mathcal{H}$ for 
Givental's symplectic vector space 
associated to the toric stack bundle 
$\mathcal{P}$. 
Let $\sigma$ be a top-dimensional cone, 
we denote Givental's symplectic vector space 
associated to the $\mathbb{T}$-fixed section
$\mathcal{P}_{\sigma}$ by $\mathcal{H}_{\sigma}$.
Let $\mathcal{H}^{tw}_{\sigma}$ and 
$\mathcal{L}^{tw}_{\sigma}$ be 
the symplectic vector space and 
Lagrangian cone associated to 
the twisted Gromov-Witten theory of 
$\mathcal{P}_{\sigma}$, 
where the twist is given by the vector bundle 
$N_{\sigma}\mathcal{P}$ and 
the $\mathbb{T}$-equivariant inverse Euler class $e^{-1}_{\mathbb{T}}$. See \cite{Tseng} for more details on twisted theory.

Let
\[
\Sigma_{top}:=\{\sigma\in \Sigma: \sigma 
\text{ is a top-dimensional cone in }\Sigma \}
\subset \Sigma
\]
be the set of top-dimensional cones in $\Sigma$. By the Atiyah-Bott localization theorem, we have an isomorphism of Chen-Ruan orbifold cohomology rings
\begin{equation}
H^{*}_{\CR,\mathbb{T}}(\mathcal{P})\otimes_{R_{\mathbb{T}}}S_{\mathbb{T}}\cong 
\bigoplus\limits_{\sigma \in \Sigma_{top}}H^{*}_{\CR}(\mathcal{P}_{\sigma})\otimes_{\mathbb{C}}S_{\mathbb{T}}.
\end{equation}
In particular, the identity class 
$1\in H^{*}_{\CR,\mathbb{T}}(\mathcal{P})$ 
corresponds to 
$\bigoplus\limits_{\sigma\in \Sigma_{top}}1_{\sigma} $,
 where $1_{\sigma}$ is the identity element in 
$H^{*}_{\CR,\mathbb{T}}(\mathcal{P}_{\sigma})$.
Furthermore, we have an isomorphism of vector spaces:
\begin{equation}\label{localization-isom}
\mathcal{H}\cong 
\bigoplus_{\sigma\in\Sigma_{top}}
\mathcal{H}_{\sigma}.
\end{equation} 
For each $f\in\mathcal{H}$ and 
$\sigma\in \Sigma_{top}$, 
let $f_{\sigma}:=f|_{\mathcal{H}_\sigma}\in \mathcal{H}_{\sigma}$ be 
the restriction of $f$ to the component 
$\mathcal{H}_{\sigma}$ of $\mathcal{H}$.
Hence $f_{\sigma}$ can also be viewed as the restriction 
of $f$ to the inertia stack $\mathcal{IP}_{\sigma}$. 
Let $f_{(\sigma,b)}:=f_\sigma|_{(\mathcal{P}_{\sigma})_{b}}$ 
be the restriction of $f_{\sigma}$ 
to the twisted sector $(\mathcal{P}_{\sigma})_{b}$ 
of $\mathcal{IP}_{\sigma}$ corresponding to the box element $b\in \bbox(\sigma)$. 

\subsection{Toric virtual localization}
We spell out explicitly the virtual localization applied to $\mathcal{P}$. Our presentation closely follows the toric case in \cite{Liu}.

The $\mathbb T$-action on $\mathcal P$ induces 
a $\mathbb T$-action on the moduli space 
$\overline{M}_{0,n+1}(\mathcal{P},\mathcal{D})$. 
The $\mathbb{T}$-fixed strata in the moduli space 
$\overline{M}_{0,n+1}(\mathcal{P},\mathcal{D})$ 
are indexed by decorated trees $\Gamma$, 
where $\Gamma$ consists of the following data. 
\begin{enumerate}

\item each vertex $v\in \Gamma$ is assigned with a top-dimensional cone 
$\sigma\in \Sigma_{top}$,  
and we denote the vertex by $v(\sigma)$.

\item each edge $e\in \Gamma$ is assigned with a codimension-$1$ cone $\tau_{e}\in\Sigma$.

\item We denote $V(\Gamma)$ to be the set of vertices of $\Gamma$,
$E(\Gamma)$ to be the set of edges of $\Gamma$. Let 
\[
F(\Gamma)=\{(e,v)\in E(\Gamma)\times V(\Gamma)| e \text{ is incident to }v\}
\]
be the set of flags in $\Gamma$.

\item Each edge $e$ is associated with a positive integer $d_{e}$ 
by the degree map $d:E(\Gamma)\rightarrow \mathbb Z_{>0}$.

\item Each flag $(e,v)$ of 
$\Gamma$ is labelled with an element $k_{(e,v)}\in G_v$,
where $G_v$ is the isotropy group
of the $\mathbb{T}$-fixed section 
$\mathcal{P}_{\sigma}$.

\item There is a marking map $s: \{1,2,\ldots,n+1\}\rightarrow V(\Gamma)$ 
that associates each marking with vertices of $\Gamma$.

\item An element $k_{j}\in G_{s(j)}$ is associated with the marking $j\in \{1,2,\ldots,n+1\}$.

\item Some compatibility conditions as in \cite[Definition 9.6]{Liu}.
\end{enumerate}

Note that the degree $\mathcal D$ is encoded in the tree $\Gamma$ by $\mathfrak D=\pi_* \mathcal D$ for each vertex $v$ and $d_e$ for each edge $e$.

We write $DT_{0,n+1}(\mathcal{P},\mathcal{D})$ for all decorated trees that contain the above data.

For a vertex $v$ in a decorated graph 
$\Gamma \in DT_{0,n+1}(\mathcal{P},\mathcal{D})$, 
we define:
\begin{itemize}
\item $S(v):=\{j\in \{1,2,\ldots,n+1\}:s(j)=v\}$, the set of markings associated to the vertex $v$.
                
\item $E(v):=\{e\in E(\Gamma):(e,v)\in F(\Gamma)\}$, the set of edges incident to the vertex $v$.
               
\item $\val(v):=|E(v)|+|S(v)|$, the valence of the vertex $v$.
\end{itemize}
We write $\mathcal{M}_{\Gamma}$ for the fixed locus of 
$\overline{M}_{0,n+1}(\mathcal{P},\mathcal{D})$ 
given by $\Gamma$,
the contribution of the Gromov-Witten invariant 
$\langle 
\gamma_{1}\bar{\psi}_{1}^{a_{1}},\ldots,\gamma_{n+1} \bar{\psi}_{n+1}^{a_{n+1}}
\rangle_{0,n+1,\mathcal{D}} $ 
from $\mathcal{M}_{\Gamma}$ is: 
\begin{align}\label{local-contr}
& c_{\Gamma}
\prod\limits_{e\in E(\Gamma)}h(e)
\prod\limits_{(e,v)\in F(\Gamma)}h(e,v)
\prod\limits_{v\in V(\Gamma)}
\left(
\prod\limits_{j:s(j)=v}\iota^{*}_{\sigma}\gamma_{j}
\right)\\
\notag & \times \prod\limits_{v\in V(\Gamma)}
\int_{
[
\overline{\mathcal{M}}^{\vec{b}(v)}_{0,val(v)}
(\mathcal{P}_{\sigma},\mathfrak D)
]^{w}
}
\frac{h(v)\prod_{j\in S(v)}\bar{\psi}^{a_{j}}_{j}}
{\prod_{e\in E(v)}(e_{\mathbb{T}}(T_{\eta(e,v)}\mathcal{C}_{e})-\bar{\psi}_{(e,v)}/r_{(e,v)})}
\end{align}
where:
\begin{itemize}
\item $c_{\Gamma}=
\frac{1}{|Aut(\Gamma)|}
\prod\limits_{e\in E(\Gamma)}
\frac{1}{d_{e}|G_{e}|}
\prod\limits_{(e,v)\in F(\Gamma)}
\frac{|G_{v}|}{r_{(e,v)}}$.

\item $G_{e}$ is the generic stabilizer of 
the toric substack bundle $\mathcal{P}_{\tau_{e}}$.

\item $r_{(e,v)}:=|\langle k_{(e,v)}\rangle|$ is the order of $k_{(e,v)}\in G_{v}$.

\item $h(e)=
\frac{e_{\mathbb{T}}(H^{1}(\mathcal{C}_{e},f^{*}_{e}T\mathcal{P})^{mov})}
{e_{\mathbb{T}}(H^{0}(\mathcal{C}_{e},f^{*}_{e}T\mathcal{P})^{mov})}$

\item $h(e,v)=
e_{\mathbb{T}}
((T_{\sigma}\mathcal{P})^{k_{(e,v)}})$

\item $h(v)=e^{-1}_{\mathbb{T}}
((N_{\sigma}\mathcal{P})_{0,val(v),\mathfrak D})$

\item $f_{e}:\mathcal{C}_{e}\rightarrow \mathcal P$ 
is a map to the toric substack bundle $\mathcal{P}_{\tau_{e}}=~^{P}\chi(\mathbf{\Sigma}/\tau_e)$. 

\item $H^{i}(\mathcal{C}_{e},f^{*}_{e}T\mathcal{P})^{mov}$ 
denotes the moving part of 
$H^{i}(\mathcal{C}_{e},f^{*}_{e}T\mathcal{P})$
with respect to the $\mathbb{T}$-action.
\item $\iota_{\sigma}:\mathcal{P}_{\sigma}\hookrightarrow \mathcal{P}$ 
is the inclusion of the fixed section $\mathcal{P}_{\sigma}$.

\item $\eta(e,v)=\mathcal{C}_e\cap \mathcal{C}_v$ is a node of $\mathcal C$ on $\mathcal{C}_{e}$, 
where $(e,v)\in F(\Sigma)$.

\item $\vec{b}(v)\in ({G_v}) ^{val(v)}$ is given by the decorations 
$k_{j},j\in S(v)$, and 
$k_{(e,v)},e\in E(v)$.
\item $(N_{\sigma}\mathcal{P})_{0,val(v),0}$ 
is the twisting bundle associated to the vector bundle 
$N_{\sigma}\mathcal{P}$ 
over the $\mathbb{T}$-fixed section $\mathcal{P}_{\sigma}$, as in \cite[Definition 2.5.10]{Tseng}.

\item $\overline{\mathcal{M}}^{\vec{b}(v)}_{0,val(v)}
(\mathcal{P}_{\sigma},\mathfrak D)$ 
is taken to be a point  
if $val(v)\leq 2$ and 
$\mathfrak D=0$. 
The twisting bundles $(N_{\sigma}\mathcal{P})_{0,val(v),0}$ 
in these unstable cases are defined to be 
$(T_{\sigma}\mathcal{P})^{k_{e,v}}$, 
as in the end of \cite[Section 9.3.4]{Liu}.
\end{itemize}

\subsection{Characterization theorem}
 For $\sigma\in \Sigma_{top}$, let $U_{k}(\sigma)$ be the character of $\mathbb{T}$ given by the restriction of the line bundle $U_{k}$ to the $\mathbb{T}$-fixed locus $\mathcal{P}_{\sigma}$. 

We will prove the following characterization result: 

\begin{thm}\label{characterization}
Let $\mathcal{P}=~^{P}\mathcal{X}(\boldsymbol{\Sigma})$ 
be a smooth toric stack bundle 
associated to an extended stacky fan 
$\boldsymbol{\Sigma} =(\textbf{N},\Sigma,\rho)$
and a $(\mathbb{C}^{\times})^{n+m}$ bundle $P\rightarrow B$. 
Let $x=(x_{1},\ldots,x_{m})$ be formal variables. 
Suppose $f$ is an element of $\mathcal{H}[[x]]$ 
satisfies $f|_{Q=q=x=0}=-1z$, 
then $f$ is a $\Lambda^{\mathbb{T}}_{nov}[[x]]$-value point 
of the Lagrangian cone $\mathcal{L}_{\mathcal{P}}$ 
if and only if it meets the following three conditions:
\begin{description}
\item[(C1)] For each
$\sigma \in \Sigma_{top}$ 
and $b\in \bbox(\sigma)$, 
the restriction $f_{(\sigma,b)}$
is a power series in $Q, q$ and $x$ 
with coefficients being elements of 
$S_{\mathbb{T}}(z)$. 
As a function in $z$, 
$f_{(\sigma,b)}$ has essential singularity at $z=0$, 
a finite order pole at $z=\infty$, simple poles at 
$z=\frac{U_{j}(\sigma)}{c}$, whenever
there exists $\sigma^{\prime}\in \Sigma$ and  $c>0 $
such that $\sigma \dagger \sigma^\prime$, 
 $j\in \sigma\setminus \sigma^{\prime}$ and  
 $\langle c\rangle= \hat{b}_{j}$. 
And $f_{(\sigma,b)}$ is regular elsewhere.

\item[(C2)] The residues of $f_{(\sigma,b)}$ 
at the simple pole $z=\frac{U_{j}(\sigma)}{c}$ satisfy the following recursion relations: 
\[
Res_{z=\frac{U_{j}(\sigma)}{c}}
f_{(\sigma,b)}(z)dz=
-q^{d_{c,\sigma,j}}
Rec(c)^{(\sigma^{\prime},b^{\prime})}_{(\sigma,b)}
f_{(\sigma^{\prime},b^{\prime})}(z)|_{z=\frac{U_{j}(\sigma)}{c}},
\]
where the {\em recursion coefficient} $Rec(c)^{(\sigma^{\prime},b^{\prime})}_{(\sigma,b)}$ associated to $(\sigma,\sigma^{\prime},b,c)$ 
is an element of $S_{\mathbb{T}}$ 
given by: 
\begin{equation*}
 Rec(c)^{(\sigma^{\prime},b^{\prime})}_{(\sigma,b)}:=\frac{1}{c}
\left(
\prod\limits_{i\in \sigma:b_{i}=0}U_{i}(\sigma)
\right)
\frac{\left(\frac{c}{U_{j}(\sigma)} \right)^{\lfloor c\rfloor}}
{\lfloor c\rfloor !}
\frac{\left(\frac{c}{U_{j}(\sigma)} \right)^{\lfloor c^{\prime}\rfloor}}
{\lfloor c^{\prime}\rfloor !}
\prod\limits_{i\in \sigma\cap\sigma^{\prime}}
\frac{\prod_{\langle a\rangle =\hat{b}_{i},a<0}(U_{i}(\sigma)+U_{j}(\sigma)\frac{a}{-c})}
{\prod_{\langle a\rangle =\hat{b}_{i},a<c_{i}}(U_{i}(\sigma)+U_{j}(\sigma)\frac{a}{-c})},
\end{equation*}
\item [(C3)] The Laurent expansion of the restriction
$f_{\sigma}$ at $z=0$ 
is a $\Lambda^{\mathbb{T}}_{nov}[[x]]$-valued point of 
the twisted Lagrangian cone $\mathcal{L}^{tw}_{\sigma}$.

\end{description}
\end{thm}
\begin{proof}
We will follow the approach in \cite{CCIT}. Let $\{\phi_{\alpha}\}$ be a basis for 
$H^{*}_{\CR,\mathbb{T}}(\mathcal{P})
\otimes_{R_{\mathbb{T}}}S_{\mathbb{T}}$ 
and $\{\phi^{\alpha}\}$ 
be its dual basis
with respect to the orbifold Poincar\'e pairing. 
Suppose $f$ is a $\Lambda^{\mathbb{T}}_{nov}[[x]]$-valued point 
on the Lagrangian cone $\mathcal{L}_{\mathcal{P}}$.
Then $f$ can be written as
\begin{equation}\label{point-in-LP}
f=
-1z+\textbf{t}(z)+
\sum\limits^{\infty}_{n=0}
\sum\limits_{\substack{d\in \NE(\X)\\ \mathfrak D\in \NE(B)}}
\sum\limits_{\alpha}
\frac{Q^{\mathfrak D}q^{d}}{n!}
\langle 
\textbf{t}(\bar{\psi}),\ldots,\textbf{t}(\bar{\psi}),
\frac{\phi_{\alpha}}{-z-\bar{\psi}}
\rangle ^{\mathbb{T}}_{0,n+1,\mathcal{D}}\phi^{\alpha}
\end{equation}
for some $\textbf{t}(z)\in \mathcal{H}_{+}[[x]]$ 
with $\textbf{t}|_{Q=q=x=0}=0$. 
Under the isomorphism (\ref{localization-isom}), 
we have that $f$ is determined by its restrictions $f_\sigma$ to $\mathcal H _\sigma$:
\[
f_{\sigma}=
-1_{\sigma}z+\textbf{t}_{\sigma}(z)+
\iota^{*}_{\sigma}
\left(
\sum\limits^{\infty}_{n=0}
\sum\limits_{\substack{d\in \NE(\X)\\ \mathfrak D\in \NE(B)}}
\sum\limits_{\alpha}
\frac{Q^{\mathfrak D}q^{d}}{n!}
\langle 
\textbf{t}(\bar{\psi}),\ldots,\textbf{t}(\bar{\psi}),
\frac{\phi_{\alpha}}{-z-\bar{\psi}}
\rangle ^{\mathbb{T}}_{0,n+1,\mathcal{D}}\phi^{\alpha}
\right)
\]
where 
$\iota_{\sigma}:\mathcal{P}_{\sigma}\rightarrow \mathcal{P}$ 
is the inclusion of the $\mathbb{T}$-fixed section. Furthermore,
let $\phi^{\alpha}_{\sigma,b}$ be 
the restriction of $\phi^{\alpha}$ to $\mathcal{IP}_{\sigma,b}$, we obtain the following sum over graphs via virtual localization in $\mathbb{T}$-equivariant cohomology:
\begin{align}\label{local-sum}
f_{(\sigma,b)}=&
-\delta_{b,0}z+\textbf{t}_{(\sigma,b)}(z)+
\sum\limits^{\infty}_{n=0}
\sum\limits_{\substack{d\in \NE(\X)\\ \mathfrak D\in \NE(B)}}
\sum\limits_{\alpha}\frac{Q^{\mathfrak D}q^{d}}{n!}
\langle \frac{\phi_{\alpha}}{-z-\bar{\psi}},\textbf{t}(\bar{\psi}),\ldots,\textbf{t}(\bar{\psi})\rangle ^{\mathbb{T}}_{0,n+1,\mathcal{D}}\phi^{\alpha}_{\sigma,b}\\
\notag =& -\delta_{b,0}z+\textbf{t}_{(\sigma,b)}(z)+
\sum\limits^{\infty}_{n=0}
\sum\limits_{\substack{d\in \NE(\X)\\ \mathfrak D\in \NE(B)}}\frac{Q^{\mathfrak D}q^{d}}{n!}
\sum\limits_{\Gamma\in DT_{0,n+1}(\mathcal{P},\mathcal{D})}\mbox{C}(\Gamma)_{\sigma,b}
\end{align}
where $\mbox{C}(\Gamma)_{\sigma,b}$ is 
the contribution 
from the $\mathbb{T}$-fixed stratum 
$\mathcal{M}_{\Gamma}\subset \overline{M}_{0,n+1}(\mathcal{P},\mathcal{D})$ 
corresponding to the decorated tree $\Gamma$. 
\[
\sum\limits_{\alpha}
\langle 
\frac{\phi_{\alpha}}{-z-\bar{\psi}},
\textbf{t}(\bar{\psi}),\ldots,\textbf{t}(\bar{\psi})
\rangle ^{\mathbb{T}}_{0,n+1,\mathcal{D}}\phi^{\alpha}_{\sigma,b}
=\sum\limits_{\Gamma\in DT_{0,n+1}(\mathcal{P},\mathcal{D})}\mbox{C}(\Gamma)_{\sigma,b}.
\]
In each decorated tree $\Gamma$, there is a distinguished vertex $v$ 
that carries the first marked point. 
We may assume that $v(\sigma)=v$ and 
the element $k_{1}$ associated with the first marking is $\hat{b}$, 
otherwise the contribution of $\Gamma$ is zero. 
There are two possibilities:
\begin{description}
\item [(A)]The irreducible component carrying the first marked point is 
a ramified cover of a $1$-dimensional orbit  
which lies in a fiber $\mathcal{X}$ 
of the toric stack bundle $\mathcal{P}\rightarrow B$. 
In this case $val(v)=2$;
\item [(B)]The irreducible component 
carrying the first marked point 
maps to a fixed section $\mathcal{P}_{\sigma}$.
\end{description}

Consider a graph $\Gamma$ of type {\bf (A)}. 
Let $e\in E(\Gamma)$ be the only edge incident to $v$. 
We denote the subgraph $\Gamma\setminus \{v,e\}$ by $\Gamma^\prime$, then $\Gamma^\prime$ is connected with $v$ through the edge $e$.
Let $v^{\prime}\in V(\Gamma^\prime)$ be the other vertex incident to $e$ and $v(\sigma^\prime)=v^\prime$
We assume the first marking of the graph $\Gamma^{\prime}$ is associated with the vertex $v^{\prime}$. 
For the fixed locus $\mathcal{M}_{\Gamma}$, 
We have $\mathcal C_e$ being a $\mathbb P^1$ toric orbifold and 
$\mathcal{C}_{e}\cong 
\mathbb{P}_{r_{(e,v)},r_{(e,v^{\prime})}}$. 
The map $f_{e}:\mathcal{C}_{e}\rightarrow \mathcal{P}$ 
satisfies 
$f_{e}(0)\in \mathcal{P}_{\sigma}$ and 
$f_{e}(\infty)\in \mathcal{P}_{\sigma^\prime}$. 
Hence, $f_{e}(\mathcal{C}_{e})$ is in a fiber of $\mathcal{P}$, 
therefore $\mathfrak D=0$, 
where $\mathcal{D}$ is the degree of $f$ and $\mathfrak D=\pi_*(\mathcal D)\in H_2(B)$.
The contribution $\text{C}(\Gamma)_{\sigma,b}$ is nontrivial only if 
\[
\phi_{\alpha}^{\sigma,\hat{b}}=
|N(\sigma)|e_{\mathbb{T}}
(N_{\sigma,b})1_{\sigma,\hat{b}}
\text{ and }\phi^{\alpha}_{\sigma,b}=[\mathcal{IP}_{\sigma,b}],
\]
where $[\mathcal{IP}_{\sigma,b}]$ 
is the fundamental class of $\mathcal{IP}_{\sigma,b}$, $N_{\sigma,b}$ is the normal bundle to 
$\mathcal{IP}_{\sigma,b}$ in $\mathcal{IP}_{b}$ 
and $1_{\sigma,\hat{b}}$ is the fundamental class of 
$\mathcal{IP}_{\sigma,\hat{b}}$ with $\hat{b}=inv(b)$.
The box element $\hat{b}\in Box(\sigma)$
is given by the restriction 
$f_{e}|_{0}:B\mu_{r_{(e,v)}}
\rightarrow \mathcal{P}_{\sigma}$. 
The morphism $f_e$ determines a rational number $c\in \mathbb Q$ and a box element 
$b^{\prime}\in Box(\sigma_{v^{\prime}})$. 
Since $\bar{\psi}_{1}=-r_{(e,v)}e_{\mathbb{T}}(T_{\eta(e,v)}\mathcal{C}_{e})$, 
using (\ref{local-contr}),
we obtain:
\begin{align*}
\mbox{C}(\Gamma)_{\sigma,b}=&
\frac{c_{\Gamma}}
{c_{\Gamma^{\prime}}}
h(e)
h(e,v)
h(e,v^\prime) \\
& \times \int_{[\overline{\mathcal{M}}^{(\hat{b},b)}_{0,2}(\mathcal{P}_{\sigma},0)]^{w}}
\frac{|N(\sigma)||e_{\mathbb{T}}(N_{\sigma,b})|}
{-z+r_{(e,v)}e_{\mathbb{T}}(T_{\eta(e,v)}\mathcal{C}_{e})}
\frac{1}
{(e_{\mathbb{T}}(T_{\eta(e,v)}\mathcal{C}_{e})-\bar{\psi}_{2}/r_{(e,v)})}
\cup h(v)\\
& \times \frac{r_{(e,v^{\prime})}}
{|N(\sigma^{\prime})||e_{\mathbb{T}}(N_{\sigma^{\prime},b^{\prime}})|}
\mbox{C}(\Gamma^{\prime})_{\sigma^{\prime},b^{\prime}}
|_{z=-r_{(e,v^{\prime})}(e_{\mathbb{T}}(T_{\eta(e,v^{\prime})}\mathcal{C}_{e})}
\end{align*}
Using \cite[(9.14)]{Liu} and the definition of $c_\Gamma$, $h(e), h(e,v),h(v)$, we write this as: 
\begin{align*}
\mbox{C}(\Gamma)_{\sigma,b}=&
\frac{|G_{v}|}
{d_{e}|G_{e}|}
h(e)
\frac{e_{\mathbb{T}}(N_{\sigma,b})}
{(-z+U_{j}(\sigma)/c)}
\mbox{C}(\Gamma^{\prime})_{\sigma^{\prime},b^{\prime}}
|_{z=-r_{(e,v^{\prime})}e_{\mathbb{T}}(T_{\eta(e,v^{\prime})}\mathcal{C}_{e})}\\
=& \frac{Rec(c)^{(\sigma^{\prime},b^{\prime})}_{(\sigma,b)}}
{(-z+U_{j}(\sigma)/c)}
\mbox{C}(\Gamma^{\prime})_{\sigma^{\prime},b^{\prime}}
|_{z=U_{j}(\sigma)/c}.
\end{align*}
Hence, the contribution to $f_{(\sigma,b)}$ from all graphs $\Gamma$ of type {\bf (A)} is:
\begin{equation}
\sum\limits_{\sigma^{\prime}:\sigma \dagger \sigma^\prime}
\sum\limits_{\substack{c\in\mathbb{Q}:c>0,\\\langle c\rangle=\hat{b}_{j}}}
q^{d_{c,\sigma,j}}
\frac{Rec(c)^{(\sigma^{\prime},b^{\prime})}_{(\sigma,b)}}
{(-z+U_{j}(\sigma)/c)}
[f_{(\sigma^{\prime},b^{\prime})}]_{z=U_{j}(\sigma)/c}.
\end{equation}
Hence we have proved $\textbf{(C2)}$, as well as $\textbf{(C1)}$.

To prove {\bf (C3)}, we define:
$t_{\sigma}(z):=
\sum\limits_{b\in \bbox(\sigma)}
t_{(\sigma,b)}(z)1_{b}$,
where 
\[
t_{(\sigma,b)}(z):=\textbf{t}_{(\sigma,b)}(z)+
\sum\limits_{\sigma^{\prime}:\sigma \dagger \sigma^\prime}
\sum\limits_{\substack{c\in\mathbb{Q}:c>0,\\\langle c\rangle=\hat{b}_{j}}}
q^{d_{c,\sigma,j}}
\frac{Rec(c)^{(\sigma^{\prime},b^{\prime})}_{(\sigma,b)}}
{(-z+U_{j}(\sigma)/c)}
[f_{(\sigma^{\prime},b^{\prime})}]_{z=U_{j}(\sigma)/c},
\]
and $t_{(\sigma,b)}(z)$ is expanded in terms of positive powers of $z$.

Then, $f_\sigma$ can be written as: 
\begin{align}\label{f_sigma}
\sum\limits_{b\in \bbox(\sigma)}
f_{(\sigma,b)}1_{b}
=& -1_{\sigma}z+t_{\sigma}(z)+
\sum\limits^{\infty}_{n=0}
\sum\limits_{\substack{d\in \NE(\X)\\ \mathfrak D\in \NE(B)}}
\sum\limits_{b\in \bbox(\sigma)}
\sum\limits_{\substack{\Gamma\in DT_{0,n+1}(\mathcal{P},\mathcal{D})\\ 
\Gamma \text{ is of type B}}}
\frac{Q^{\mathfrak D}q^{d}}{n!}\mbox{C}(\Gamma)_{\sigma,b}.
\end{align}
Then, we consider the contribution given by decorated trees $\Gamma$ of type {\bf (B)} 
such that $val(v)=l$, where $v$ is the distinguished vertex. The element $k_{1}$ 
associated to the first marking is $\hat{b}\in \bbox(\sigma)$. 
By integrating over all the factors 
$\overline{\mathcal{M}}^{\vec{b^{\prime}}}_{0,val(v^{\prime})}
(\mathcal{P}_{\sigma^{\prime}})$ except those 
associated with the distinguished vertex $v$, we can write these contributions as:
\[
\sum\limits_{\alpha}
\frac{1}
{\mbox{Aut}(\Gamma_{2},\ldots,\Gamma_{l})}
\left(
\int_{[\overline{\mathcal{M}}_{0,l}^{\hat{b},b^{2},\ldots,b^{l}}(\mathcal{P}_{\sigma},\mathfrak D)]^{w}}
\frac{\phi_{\alpha}^{\sigma,\hat{b}}}
{-z-\bar{\psi}}
\cup p_{2}(\textbf{t},\bar{\psi}_{2})
\cup\ldots\cup
p_{l}(\textbf{t},\bar{\psi}_{l})
\cup e^{-1}_{\mathbb{T}}((N_{\sigma}\mathcal{P})_{0,l,\mathfrak D})
\right ) 
\phi^{\alpha}_{\sigma,b}
\]
for some box elements $b^{2},\ldots,b^{l}\in \bbox(\sigma)$ 
and some polynomials $p_{i}(\textbf{t},\bar{\psi}_{i})$
in $t_{0},t_{1},\ldots,Q,q$ and $\bar{\psi}_{i}$.
The graph $\Gamma$ is obtained from joining type {\bf (A)} subgraphs 
$\Gamma_{2},\ldots,\Gamma_{l}$ at the vertex $v$. 
More precisely, $\Gamma_{i}$, for $2\leq i \leq l$, is of type {\bf (A)} and satisfies one of the following:
\begin{itemize}
\item $\Gamma_{i}$ consists of the distinguished vertex $v$ 
and two markings with the first marking coincides with the first marking of $\Gamma$. $val(v)=2$.

\item $\Gamma_{i}$ contains the distinguished vertex $v$ 
with exactly one marking that coincides with the first marking of $\Gamma$
and exactly one edge $e_{i}$ connecting $v$ with the rest of the graph. $val(v)=2$.
\end{itemize}
If $\Gamma_{i}$ consists of one vertex with two markings, 
then $p_{i}(\textbf{t},\bar{\psi}_{i})=\textbf{t}_{(\sigma,b^{i})}(\bar{\psi}_{i})$. Otherwise, 
\[
p_{i}(\textbf{t},\bar{\psi}_{i})=Q^{\mathfrak D}q^{d_{i}}\mbox{C}(\Gamma_{i})_{\sigma,b^{i}}|_{z=\bar{\psi}_{i}}
\]
where $d_{i}$ is the degree from the subgraph $\Gamma_{i}$. 
Summing over the contribution $\mbox{C}(\Gamma)_{\sigma,b}$ over all $\Gamma$ such that $val(v)=l$ gives the contribution
\[
\sum\limits_{\alpha}
\frac{1}{(l-1)!}
\left(
\int_{[\overline{\mathcal{M}}_{0,l}^{\hat{b},b^{2},\ldots,b^{l}}(\mathcal{P}_{\sigma},\iota_{\sigma}^{*}\mathcal{D})]^{w}}
\frac{\phi_{\alpha}^{\sigma,\hat{b}}}
{-z-\bar{\psi}}
\cup t_{\sigma}(\bar{\psi}_{2})
\cup\ldots\cup
t_{\sigma}(\bar{\psi}_{l})
\cup e^{-1}_{\mathbb{T}}((N_{\sigma}\mathcal{P})_{0,l,\mathfrak D})
\right ) 
\phi^{\alpha}_{\sigma,b}
\]
in (\ref{f_sigma}).
Hence, we have
\begin{align*}
f_{\sigma}= 
-1_{\sigma}z+t_{\sigma}(z)+\sum\limits_{l=1}^{\infty}
\sum\limits_{\mathfrak D\in \NE(\mathcal{P}_{\sigma})}
\sum\limits_{b\in \bbox(\sigma)}
\sum\limits_{\alpha}
\frac{1}{(l-1)!}
\langle \frac{\phi_{\alpha}^{\sigma,\hat{b}}}
{-z-\bar{\psi}}
,t_{\sigma}(\psi),
\ldots,t_{\sigma}(\psi)
\rangle^{\tw}_{0,l,\mathfrak D}\phi^{\alpha}_{\sigma,b}
\in \mathcal L^{tw}_\sigma
\end{align*}
i.e., the Laurent expansion at $z=0$ 
of $f_{\sigma}$ lies in the twisted Lagrangian cone $\mathcal{L}^{tw}_{\sigma}$. 
Thus we have proved $\textbf{(C3)}$.

To prove the other direction of the theorem, we assume that $f\in\mathcal{H}[[x]]$ 
with $f|_{Q=q=x=0}=-1z$ satisfies conditions {\bf (C1), (C2)}, and {\bf (C3)}. 
Then, from conditions {\bf(C1)} and {\bf(C2)}, we obtain that: 
\begin{equation}\label{f-sigma}
f_{\sigma}=
-1_{\sigma}z+\textbf{t}_{\sigma}+
\sum\limits_{b\in \bbox(\sigma)}1_{b}
\sum\limits_{\sigma^{\prime}:\sigma \dagger \sigma^\prime}
\sum\limits_{\substack{c\in\mathbb{Q}:c>0,\\\langle c\rangle=\hat{b}_{j}}}
q^{d_{c,\sigma,j}}
\frac{RC(c)^{(\sigma^{\prime},b^{\prime})}_{(\sigma,b)}}
{(-z+U_{j}(\sigma)/c)}
[f_{(\sigma^{\prime},b^{\prime})}]_{z=U_{j}(\sigma)/c}+O(z^{-1})
\end{equation}
for some $\textbf{t}_{\sigma}\in \mathcal{H}_{\sigma,+}[[x]]$ 
satisfying $\textbf{t}_{\sigma}|_{Q=q=t=0}=0$. 
The remainder $O(z^{-1})$ is a formal power series in $Q$, $q$ and $x$ 
with coefficients in $z^{-1}S_{\mathbb{T}}[z^{-1}]$.
Let $F$ be a $\Lambda^{\mathbb{T}}_{nov}[[x]]$-valued point 
on $\mathcal{L}_{\mathcal{P}}$ defined by 
(\ref{point-in-LP}) with $\textbf{t}=\tau$, where 
$\tau\in \mathcal{H}_{+}[[x]]$ is 
the unique element such that its restriction to 
$\mathcal{IP}_{\sigma}$ is $\textbf{t}_{\sigma}$. 
Then, we know that $F$ and $f$ both 
satisfy conditions {\bf (C1-C3)}, 
and they have the same restriction $\textbf{t}_{\sigma}$
in $\mathcal {IP}_\sigma$. 
Hence, it remains to show that $f$ can be uniquely determined by  
the set of elements $\{\textbf{t}_\sigma\}_{\sigma\in \Sigma_{top}}$.

To prove the uniqueness, we use induction on the degree with respect to $Q, q$ and $x$. 
Choose a K\"ahler class $\omega$ of $\mathcal{P}$, recall that the degree  
of the monomial $Q^{\mathfrak D}q^{d}x_{1}^{k_{1}}\cdots x_{m}^{k_{m}}$,
can be defined as $\langle \mathcal{D},\omega\rangle+\sum\limits_{i=1}^{m}k_{i}$.  
Let $\kappa_{0}$ denote the minimal degree of a non-trivial stable map to $\mathcal P$. 
Suppose that $f$ is uniquely determined from the collection 
$\{t_\sigma\}_{\sigma\in \Sigma_{top}}$
up to order $\kappa$. 
By the isomorphism (\ref{localization-isom}), we know that 
$f$ is uniquely determined by the collection of its restrictions $\{f_\sigma\}$,
hence to show $f$ is determined up to order $\kappa+\kappa_{0}$, 
we just need to show $f_{\sigma}$ is determined up to order $\kappa+\kappa_{0}$. 
We know by (\ref{f-sigma}) that 
$f_{\sigma}$ is determined up to order $\kappa+\kappa_{0}$ 
except for the remainder $O(z^{-1})$. 
On the other hand, 
since the Laurent expansion at $z=0$ of $f_{\sigma}$ lies in $\mathcal{L}^{tw}_{\sigma}$, 
equation (\ref{point-in-LP}) implies that 
the higher order terms $O(z^{-1})$ of $z^{-1}$ is also uniquely determined up to order $\kappa+\kappa_{0}$. 
The proof is completed.
\end{proof}

\section{Proof of the main theorem}\label{sec:proof_main_thm}
To prove Theorem \ref{main-theorem}, 
it suffices to show the $S$-extended $I$-function 
$I^{S}_{\cP}(z,t,\tau,q,x,Q)$ 
satisfies conditions {\bf (C1)-(C3)} in Theorem \ref{characterization}. 
Recall that the definition of $I^{S}_{\cP}(z,t,\tau,q,x,Q)$ is in Definition \ref{defn:I_func}.
Let $I^{S}_{\sigma}$ and $I^{S}_{(\sigma,b)}$ denote the restrictions of $I^{S}_{\cP}(z,t,\tau,q,x,Q)$ to the inertia stack $\mathcal{IP}_{\sigma}$ and the component $(\mathcal{P}_{\sigma})_{b}$ of the inertia stack $\mathcal{IP}_{\sigma}$ respectively.

\subsection{Condition {\bf (C1)}: Poles of $I$-function}
By Definition \ref{defn:I_func}, we have
\begin{align}\label{I-function-restriction}
I^{S}_{(\sigma,b)}=&
 e^{\sum^{n}_{i=1}U_{i}(\sigma)t_{i}/z}
\sum\limits_{\mathfrak D\in \NE(B)}
\sum\limits_{\lambda\in \Lambda E^{S}_{b}}
J_{\mathfrak D}(z,\tau)Q^{\mathfrak D}\tilde{q}^{\lambda}e^{\lambda t}\\
& \times \left( 
\prod\limits_{i \in \sigma}
\frac{\prod_{\langle a\rangle =\langle \lambda_{i}-c_{1}(L_{i})\cdot\mathfrak D\rangle,a\leq 0}(U_{i}(\sigma)+az)}
{\prod_{\langle a\rangle =\langle \lambda_{i}-c_{1}(L_{i})\cdot\mathfrak D\rangle,a\leq \lambda_{i}-c_{1}(L_{i})\cdot\mathfrak D}(U_{i}(\sigma)+az)}
\right)
\notag \left( \prod\limits_{i \not\in \sigma}\frac{\prod_{\langle a\rangle =0,a\leq 0}(az)}{\prod_{\langle a\rangle =0,a\leq \lambda_{i}-c_{1}(L_{i})\cdot\mathfrak D}(az)}\right)
\end{align}
where we identify the top-dimensional  cone $\sigma$ 
with the index set of its 1-cones and consider $\sigma\subset
\{
1,\ldots,n
\} \text{ as a subset of } 
\{
1,\ldots,n+m
\}
$.
Note that $U_{i}(\sigma)=0$ for 
$i\not\in \sigma$. 
We also need to have 
$\lambda_{i}-c_{1}(L_{i})\cdot\mathfrak D\geq 0$ 
for $i\not\in \sigma$, 
otherwise the contribution is zero. 
Therefore, we can see that 
$I^{S}_{(\sigma,b)}$ has essential singularity at 
$z=0$ and a finite order pole at $z=\infty$ and simple poles at
\[
z=-U_{i}(\sigma)/a\quad \text{with} 
\quad 0<a\leq \lambda_{i}-c_{1}(L_{i})\cdot\mathfrak D, 
\langle a\rangle=\langle \lambda_{i}-c_{1}(L_{i})\cdot\mathfrak D\rangle=
\langle \lambda_{i}\rangle=
\hat{b}_{i},i\in \sigma,
\]
for $\lambda\in \Lambda E^{S}_{b}$ contributing to the sum. 
To see it satisfies {\bf (C1)} of Theorem \ref{characterization}, 
it remains to invoke the following lemma, which is proved in \cite[Section 7.1]{CCIT}:
\begin{lemma} 
Consider a top dimensional cone $\sigma\in\Sigma_{top}$, if $\lambda_{i_{0}}>0$ for some $i_{0}\in \sigma$, 
then there exists another top-dimensional cone $\sigma^{\prime}\in\Sigma_{top}$, 
such that $\sigma \dagger \sigma^\prime$ and 
$i_{0}\in \sigma\setminus \sigma^{\prime}$. 
\end{lemma}

\subsection{Condition {\bf (C2)}: Recursion relations} 
let $\sigma,\sigma^{\prime}\in\Sigma_{top}$ be 
top-dimensional cones with $\sigma \dagger \sigma^\prime$. 
Let $b\in \bbox(\sigma)$ and 
fix a positive rational number $c$ 
such that $\langle c\rangle=\hat{b}_{j}$. 
We study the residue of $I^{S}_{(\sigma,b)}$ 
at $z=-U_{j}(\sigma)/c$. Write
\[
\Delta_{\lambda,i,\sigma,\mathfrak D}(z):=
\frac{\prod_{\langle a \rangle=\langle \lambda_{i}\rangle, a\leq 0}
U_{i}(\sigma)+az}
{\prod_{\langle a \rangle=\langle \lambda_{i}\rangle, a\leq \lambda_{i}-c_{1}(L_{i})\cdot\mathfrak D}
U_{i}(\sigma)+az}
\]
for $\lambda\in \Lambda^{S}$, $1\leq i\leq n+m$, and $\mathfrak D\in \NE(B)$. 
The residue of (\ref{I-function-restriction}) is given by:
\begin{equation}\label{I-function-residue}
 e^{\frac{\sum^{n}_{i=1}U_{i}(\sigma)t_{i}}
{-U_{j}(\sigma)/c}}
\frac{1}{c}
\sum\limits_{\mathfrak D\in \NE(B)}
\sum\limits_{\substack{\lambda\in \Lambda E^{S}_{b}\\ \lambda_{j}\geq c}}
J_{\mathfrak D}(-U_{j}(\sigma)/c,\tau)Q^{\mathfrak D}\tilde{q}^{\lambda}e^{\lambda t}
\frac{\prod_{i:i\neq j}\Delta_{\lambda,i,\sigma,\mathfrak D}(-U_{j}(\sigma)/c)}
{\prod_{\substack{0<a\leq \lambda_{j}-c_{1}(L_{j})\cdot\mathfrak D,\langle a \rangle =\langle \lambda_{j}\rangle \\ a\neq c}
}
(U_{j}(\sigma)-a\frac{U_{j}(\sigma)}{c})
}.
\end{equation}

Consider the change of variables 
\[
\lambda=\lambda^{\prime}+d_{c,\sigma,j}
\]
where $\lambda^{\prime}\in \Lambda^{S}_{b^{\prime}}$. We write 
\[
c_{i}=
D_{i}\cdot d_{c,\sigma,j}, \text{ for }1\leq i \leq n; 
\quad c_{j}=c; 
\quad c_{j^{\prime}}=c^{\prime}; 
\quad c_{i}=0, \text{ for }n+1\leq i \leq n+m.
\]
For $1\leq i \leq n$, 
consider the representable morphism 
$f:\mathbb{P}_{r_{1},r_{2}}\rightarrow \mathcal{P}$ given by 
Proposition \ref{toric-morphism}
with $f(0)\in \mathcal{P}_{\sigma}$ and 
$f(\infty)\in \mathcal{P}_{\sigma^\prime}$, then applying the localization formula, we have
\[
c_{i}=D_{i}\cdot d_{c,\sigma,j}=\int_{\mathbb{P}_{r_{1},r_{2}}}f^{*}D_{i}
=\int_{\mathbb{P}_{r_{1},r_{2}}}^{\mathbb{T}}f^{*}U_{i}
=\frac{U_{i}(\sigma)}
{U_{j}(\sigma)/c}+
\frac{U_{i}(\sigma^{\prime})}{-U_{j}(\sigma^{\prime})/c^{\prime}}
=\frac{U_{i}(\sigma)}
{U_{j}(\sigma)/c}+
\frac{U_{i}(\sigma^{\prime})}{-U_{j}(\sigma)/c}
\]
Hence we obtain 
\begin{equation}\label{U-i-sigma}
U_{i}(\sigma)=U_{i}(\sigma^{\prime})+\frac{c_{i}}{c}U_{j}(\sigma).
\end{equation}
Hence, by equation (\ref{U-i-sigma}), we have the following three equations
\begin{equation}
\frac{\sum^{n}_{i=1}U_{i}(\sigma)t_{i}}
{-U_{j}(\sigma)/c}+\lambda t
=\frac{\sum^{n}_{i=1}U_{i}(\sigma^{\prime})t_{i}}
{-U_{j}(\sigma)/c}+\lambda^{\prime} t;
\end{equation}
\begin{equation}
\Delta_{\lambda,i,\sigma,\mathfrak D}\left(-\frac{U_{j}(\sigma)}{c}\right)
=\Delta_{\lambda^{\prime},i,\sigma^{\prime},\mathfrak D}\left(-\frac{U_{j}(\sigma)}{c}\right)
\frac{\prod_{a\leq 0, \langle a\rangle =
\langle \lambda_{i}\rangle}(U_{i}(\sigma)-\frac{a}{c}U_{j}(\sigma))}
{\prod_{a\leq c_{i}, \langle a\rangle =
\langle \lambda_{i}\rangle}(U_{i}(\sigma)-\frac{a}{c}U_{j}(\sigma))}, \quad 
\text{for } i\neq j;
\end{equation}
\begin{equation}
\prod\limits_{\substack{0<a\leq \lambda_{j}-c_{1}(L_{j})\cdot\mathfrak D,\\
\langle a \rangle =
\langle \lambda_{j} \rangle, a\neq c}}
\left(
U_{j}(\sigma)-a\frac{U_{j}(\sigma)}{c}
\right)
=\prod\limits_{\substack{-c<a\leq \lambda^{\prime}_{j}-c_{1}(L_{j})\cdot\mathfrak D,\\\langle a \rangle =\langle \lambda_{j} \rangle, a\neq 0}}
\left(
-a\frac{U_{j}(\sigma)}{c}
\right).
\end{equation}

Applying the above three equations we see that (\ref{I-function-residue}), 
the residue of $I^S_{(\sigma,b)}$ at $z=-\frac{U_i(\sigma)}{c}$, is given by:
\begin{align*}
& e^{\frac{\sum^{n}_{i=1}U_{i}(\sigma^{\prime})t_{i}}
{-U_{j}(\sigma)/c}}
\frac{1}{c}
\sum\limits_{\mathfrak D\in \NE(B)}
\sum\limits_{\substack{\lambda^{\prime}\in \Lambda E^{S}_{b^{\prime}}\\ \lambda^{\prime}_{j}\geq 0}}
J_{\mathfrak D}(-U_{j}(\sigma)/c,\tau)Q^{\mathfrak D}\tilde{q}^{\lambda^{\prime}}q^{d_{c,\sigma,j}}e^{\lambda^{\prime} t}
 \frac{\prod_{i:i\neq j}\Delta_{\lambda^{\prime},i,\sigma^{\prime},\mathfrak D}(-U_{j}(\sigma)/c)}
{
\prod\limits_{
\substack{0<a\leq \lambda^{\prime}_{j}-c_{1}(L_{j})\cdot\mathfrak D,\\\langle a \rangle =\langle \lambda_{j} \rangle, a\neq 0}
}
\left(
U_{j}(\sigma)-a\frac{U_{j}(\sigma)}{c}
\right)
}\\
& \times \prod\limits_{i:i\neq j}
\frac{\prod_{a\leq 0, \langle a\rangle =\langle \lambda_{i}\rangle}(U_{i}(\sigma)-\frac{a}{c}U_{j}(\sigma))}
{\prod_{a\leq c_{i}, \langle a\rangle =\langle \lambda_{i}\rangle}(U_{i}(\sigma)-\frac{a}{c}U_{j}(\sigma))}\\
=& q^{d_{c,\sigma,j}}\frac{1}{c}\frac{1}
{\prod\limits_{\substack{0<a<c,\\a\in \mathbb{Z}}}\left(a\frac{U_{j}(\sigma)}{c}\right)}
\prod\limits_{i\in \sigma^{\prime}}
\frac{\prod_{a\leq 0, \langle a\rangle =
\langle \lambda_{i}\rangle}(U_{i}(\sigma)-\frac{a}{c}U_{j}(\sigma))}
{
\prod_{a\leq c_{i}, \langle a\rangle =
\langle \lambda_{i}\rangle}
(U_{i}(\sigma)-\frac{a}{c}U_{j}(\sigma))
}
I^{S}_{(\sigma^{\prime},b^{\prime})}|_{z=-U_{j}(\sigma)/c}. 
\end{align*}
By a direct computation, we obtain
\[
\frac{1}{c}\frac{1}
{\prod\limits_{\substack{0<a<c,\\a\in \mathbb{Z}}}\left(a\frac{U_{j}(\sigma)}{c}\right)}
\prod\limits_{i\in \sigma^{\prime}}
\frac{\prod_{a\leq 0, \langle a\rangle =\langle \lambda_{i}\rangle}(U_{i}(\sigma)-\frac{a}{c}U_{j}(\sigma))}
{\prod_{a\leq c_{i}, \langle a\rangle =\langle \lambda_{i}\rangle}(U_{i}(\sigma)-\frac{a}{c}U_{j}(\sigma))}
=Rec(c)^{(\sigma^{\prime},b^{\prime})}_{(\sigma,b)}.
\]
This proves the recursion for the $S$-Extended $I$-function.

\subsection{Condition {\bf (C3)}: Restriction to fixed points.}
Consider a top dimensional cone $\sigma\in\Sigma_{top}$, 
we need to show that 
$I^{S}_{\sigma}$ lies on the Lagrangian cone $\mathcal{L}^{tw}_{\sigma}$. 
We will need to use the decomposition theorem of Gromov-Witten theory of $\mu$-gerbe over the base $B$ as in \cite{AJT09}.

By Definition \ref{defn:I_func}, we have
\begin{align}\label{I^S_sigma}
I^{S}_{\sigma} = &  e^{\sum^{n}_{i=1}U_{i}(\sigma)t_{i}/z}
\sum\limits_{\mathfrak D\in \NE(B)}
\sum\limits_{\substack{\lambda\in \Lambda^{S}_{\sigma}\\ \lambda_{i}-c_{1}(L_{i})\cdot\mathfrak D\in\mathbb{Z}_{\geq 0} \text{ if } i \not\in\sigma}}
J_{\mathfrak D}(z,\tau)Q^{\mathfrak D}\tilde{q}^{\lambda}e^{\lambda t}\\
& \times \left(
 \prod\limits_{i \in \sigma}
\frac{\prod_{\langle a \rangle =\langle \lambda_{i}-c_{1}(L_{i})\cdot\mathfrak D\rangle,a\leq 0}(U_{i}(\sigma)+az)}
{\prod_{\langle a\rangle =\langle \lambda_{i}-c_{1}(L_{i})\cdot\mathfrak D\rangle,a\leq \lambda_{i}-c_{1}(L_{i})\cdot\mathfrak D}(U_{i}(\sigma)+az)}
\right)
\notag 
\left(
 \prod\limits_{i \not\in \sigma}
\frac{\prod_{\langle a\rangle =0,a\leq 0}(az)}
{\prod_{\langle a\rangle =0,a\leq \lambda_{i}-c_{1}(L_{i})\cdot\mathfrak D}(az)}
\right)
1_{v^{S}(\lambda)}
\end{align}
where $1_{v^{S}(\lambda)}\in H^{*}_{\CR}(\mathcal{P}_{\sigma})$ 
is the identity class on the twisted sector of $\mathcal{P}_{\sigma}$
corresponding to the box element $v^{S}(\lambda)\in \bbox(\sigma)$. 
By string equation, 
$\mathcal{L}^{tw}_{\sigma}$ is invariant under multiplication by 
$ e^{\sum^{n}_{i=1}U_{i}(\sigma)t_{i}/z}$, 
hence we can remove this factor from (\ref{I^S_sigma}).

Let $\pi(\sigma)$ be the quotient map $\textbf{N} \rightarrow \textbf{N}(\sigma)$.
We have
\[
v^{S}(\lambda)=
\sum\limits_{j\in\sigma}
\lceil \lambda_{j}\rceil \rho_{j}+
\sum\limits_{i\not\in\sigma,i\leq n}
\lambda_{i}\rho_{i}+
\sum\limits_{i=1}^{m}
\lambda_{n+i}s_{i}\equiv 
\sum\limits_{i\not\in\sigma}\lambda_{i}b^{i}_\sigma 
\quad \text{mod }\textbf{N}_{\sigma},
\]
where
\[
b^{i}_\sigma=\left\{
\begin{array}{lr}
\pi(\sigma)(\rho_i), \quad 1\leq i \leq n;\\
\pi(\sigma)(s_{i-n}), \quad n+1\leq i\leq n+m.
\end{array}
\right.
\]
We also introduce variables 
$\{q_{i}\}_{i\not\in \sigma}$ 
and the change of variables:
\[
Q^\mathfrak D\tilde{q}^{\lambda}e^{\lambda t}=
(\prod\limits_{i\not\in \sigma}q_{i}^{\lambda_{i}-c_{1}(L_{i})\cdot\mathfrak D})Q^\mathfrak D
\]
It remains to show that 
\begin{align}\label{restriction-sigma}
\sum\limits_{\mathfrak D\in \NE(B)}
\sum\limits_{
\substack{
\lambda\in \Lambda^{S}_{\sigma}\\ 
\lambda_{i}-c_{1}(L_{i})\cdot\mathfrak D\in\mathbb{Z}_{\geq 0} 
\text{ if } i \not\in\sigma
}
}
Q^{\mathfrak D}J_{\mathfrak D}(z,\tau)
\left( 
\prod\limits_{i \in \sigma}
\frac{\prod_{\langle a \rangle =\langle \lambda_{i}-c_{1}(L_{i})\cdot\mathfrak D\rangle,a\leq 0}(U_{i}(\sigma)+az)}
{\prod_{\langle a\rangle =\langle \lambda_{i}-c_{1}(L_{i})\cdot\mathfrak D\rangle,a\leq \lambda_{i}-c_{1}(L_{i})\cdot\mathfrak D}(U_{i}(\sigma)+az)}
\right)\\
\notag 
\times \left( 
\prod\limits_{i \not\in \sigma}
\frac{q_{i}^{\lambda_{i}-c_{1}(L_{i})\cdot\mathfrak D}\prod_{\langle a\rangle =0,a\leq 0}(az)}
{\prod_{\langle a\rangle =0,a\leq \lambda_{i}-c_{1}(L_{i})\cdot\mathfrak D}(az)}
\right)
1_{\sum_{i\not\in \sigma}\lambda_{i}b^{i}_\sigma}
\end{align}
is a $S_{\mathbb{T}}[[q]]$-valued point on the twisted Lagrangian cone $\mathcal{L}^{tw}_{\sigma}$ of $\mathcal P_{\sigma}$.

By construction, $\mathcal P_\sigma$ is a product of root gerbes over the base $B$. 
\begin{lemma}
For each top dimensional cone $\sigma$, the series
\begin{equation}\label{P-sigma-function}
\sum\limits_{\mathfrak D\in \NE(B)} 
\sum\limits_{b\in \bbox(\sigma)}
Q^{\mathfrak D}
J_{\mathfrak D}(z,\tau)1_{b}
\end{equation}
lies on the Lagrangian cone $\mathcal{L}$ 
of the untwisted theory of $\mathcal{P}_{\sigma}$
\end{lemma}
\begin{proof}
By the definition of J-function, we have
\begin{align*}
\sum\limits_{\mathfrak D\in \NE(B)} 
\sum\limits_{b\in \bbox(\sigma)}
Q^{\mathfrak D}
J_{\mathfrak D}(z,\tau)1_{b}
&=\sum\limits_{b\in \bbox(\sigma)}
(z+\tau+\sum_{n,\mathfrak D}\sum_\alpha\frac{Q^\mathfrak D}{n!}\langle \frac{\phi_\alpha}{z-\bar{\psi}},\tau,\ldots,\tau\rangle_{0,n+1,\mathfrak D}^B\phi^\alpha)1_{b}\\
&= z\sum\limits_{b\in \bbox(\sigma)}
\sum_{n,\mathfrak D}\sum_\alpha\frac{Q^\mathfrak D}{n!}\langle \frac{\phi_\alpha}{z-\bar{\psi}},1,\tau,\ldots,\tau\rangle_{0,n+2,\mathfrak D}^B\phi^\alpha\otimes 1_{b},
\end{align*}
where the second equation is from string euqation. By \cite[Proposition 5.2]{Jiang}, the Chen-Ruan cohomology ring of the toric stack bundle $\mathcal P_\sigma$ is given by
\begin{equation}\label{QH*-gerbe}
H^*_{CR}(\mathcal P_\sigma;\mathbb Q)\cong H^*(B;\mathbb Q)\otimes H^*_{CR}(BG_\sigma;\mathbb Q)
\end{equation}
where $G_\sigma$ is the isotropy group at generic points of $\mathcal P_\sigma$. Recall that $\tau$ takes value in $H^*(B,\mathbb Q)\cong H^*(\mathcal P_{\sigma},\mathbb Q)\subset H^*_{CR}(\mathcal P_\sigma,\mathbb Q)$. By \cite[Theorem 7.1]{AJT10}, 
\begin{align*}
\langle \frac{\phi_\alpha}{z-\bar{\psi}},1,\tau,\ldots,\tau\rangle_{0,n+2,\mathfrak D}^B=|G_\sigma|\langle \frac{\phi_\alpha\otimes 1_{\hat b}}{z-\bar{\psi}},1_{\tilde{b}},\tau,\ldots,\tau\rangle_{0,n+2,\mathfrak D}^{\mathcal P_\sigma}
\end{align*}
where $\hat{b}=inv(b)$ is the involution of $b$ and $\tilde b$ is uniquely determined by the $\mathfrak D$-admissible condition defined in \cite[Definition 3.3, Remark 3.7]{AJT09} and \cite[Definition 4.5]{AJT10}. Note that $\tilde b$ depends on $b$ and $\mathfrak D$. Moreover, by \cite[Theorem 4.4]{AJT09} and \cite[Theorem 7.1]{AJT10}, for each $\check b\in \bbox(\sigma)\setminus{\tilde b}$, the invariant
$\langle \frac{\phi_\alpha\otimes 1_{\hat b}}{z-\bar{\psi}},1_{\check{b}},\tau,\ldots,\tau\rangle_{0,n+2,\mathfrak D}^{\mathcal P_\sigma}$
vanishes. Therefore
\[
\langle \frac{\phi_\alpha\otimes 1_{\hat b}}{z-\bar{\psi}},1_{\tilde{b}},\tau,\ldots,\tau\rangle_{0,n+2,\mathfrak D}^{\mathcal P_\sigma}=\langle \frac{\phi_\alpha\otimes 1_{\hat b}}{z-\bar{\psi}},\sum_{\underline b\in \bbox(\sigma)}1_{\underline b},\tau,\ldots,\tau\rangle_{0,n+2,\mathfrak D}^{\mathcal P_\sigma},
\]
for every $b$ and $\mathfrak D$. Notice that $\sum_{\underline b\in \bbox(\sigma)}1_{\underline b}$ does not depend on $b$ or $\mathfrak D$.
Then, (\ref{P-sigma-function}) can be written as
\begin{align*}
z|G_\sigma|\sum\limits_{b\in \bbox(\sigma)}
\sum_{n,\mathfrak D}\sum_\alpha\frac{Q^\mathfrak D}{n!}\langle \frac{\phi_\alpha\otimes 1_{\hat b}}{z-\bar{\psi}},\sum_{\underline b\in \bbox(\sigma)}1_{\underline b},\tau,\ldots,\tau\rangle_{0,n+2,\mathfrak D}^{\mathcal P_\sigma}\phi^\alpha \otimes1_{b}
\end{align*}
By (\ref{QH*-gerbe}), $\{\phi^\alpha\otimes 1_b\}_{b\in \bbox{\sigma},\alpha}$ forms a basis of $H^*_{CR}(\mathcal P_\sigma;\mathbb Q)$, so we have
\begin{align*}
& z|G_\sigma|\sum\limits_{b\in \bbox(\sigma)}
\sum_{n,\mathfrak D}\sum_\alpha\frac{Q^\mathfrak D}{n!}\langle \frac{\phi_\alpha\otimes 1_{\hat b}}{z-\bar{\psi}},\sum_{\underline b\in \bbox(\sigma)}1_{\underline b},\tau,\ldots,\tau\rangle_{0,n+2,\mathfrak D}^{\mathcal P_\sigma}\phi^\alpha \otimes1_{b}\\
=& z|G_\sigma| S_{\mathcal P_\sigma}(\tau,z)(\sum_{\underline b\in \bbox(\sigma)}1_{\underline b}),
\end{align*} 
by the definition of the $S$-operator (\ref{S-operator}). Hence, we conclude that (\ref{P-sigma-function}) lies on the Lagrangian cone $\mathcal{L}$ 
of the untwisted theory of $\mathcal{P}_{\sigma}$.
\end{proof}

Notice that we have the following identity
\begin{align*}
\sum\limits_{
\substack{
\lambda\in \Lambda^{S}_{\sigma}\\ 
\lambda_{i}-c_{1}(L_{i})\cdot\mathfrak D\in\mathbb{Z}_{\geq 0} 
\text{ if } i \not\in\sigma
}
}
\left( 
\prod\limits_{i \not\in \sigma}
\frac{q_{i}^{\lambda_{i}-c_{1}(L_{i})\cdot\mathfrak D}\prod_{\langle a\rangle =0,a\leq 0}(az)}
{\prod_{\langle a\rangle =0,a\leq \lambda_{i}-c_{1}(L_{i})\cdot\mathfrak D}(az)}
\right)
1_{\sum_{i\not\in \sigma}\lambda_{i}b^{i}_\sigma}= \sum\limits_{b\in \bbox(\sigma)}
\exp(\sum\limits_{i\not\in \sigma}q_{i}/z)1_{b}.
\end{align*}

Hence, we have
\begin{align}\label{J-function-gerbe}
& \sum\limits_{\mathfrak D\in \NE(B)}
\sum\limits_{
\substack{
\lambda\in \Lambda^{S}_{\sigma}\\ 
\lambda_{i}-c_{1}(L_{i})\cdot\mathfrak D\in\mathbb{Z}_{\geq 0} 
\text{ if } i \not\in\sigma
}
}
Q^\mathfrak D
J_{\mathfrak D}(z,\tau)
\left( 
\prod\limits_{i \not\in \sigma}
\frac{q_{i}^{\lambda_{i}-c_{1}(L_{i})\cdot\mathfrak D}\prod_{\langle a\rangle =0,a\leq 0}(az)}
{\prod_{\langle a\rangle =0,a\leq \lambda_{i}-c_{1}(L_{i})\cdot\mathfrak D}(az)}
\right)
1_{\sum_{i\not\in \sigma}\lambda_{i}b^{i}_\sigma}\\
\notag =& 
\sum\limits_{\mathfrak D\in \NE(B)} 
\sum\limits_{b\in \bbox(\sigma)}
Q^\mathfrak D
\exp(\sum\limits_{i\not\in \sigma}q_{i}/z)J_{\mathfrak D}(z,\tau)1_{b}.
\end{align}
lies on the untwisted Lagrangian cone $\mathcal{L}$ of $\mathcal P_{\sigma}$.

We will need to use Tseng's orbifold quantum Riemann-Roch theorem in \cite{Tseng}
to prove (\ref{restriction-sigma}) lies in the 
twisted Lagrangian cone $\mathcal L^{tw}_{\sigma}$.
We recall some notations here:

Let $V$ be the direct sum of 
$d$ vector bundles $V^{(j)}$, for $1\leq j \leq d$, and 
consider a universal multiplicative characteristic class:
\[
\textbf{c}(V)=
\prod\limits_{j=1}^{d}
\exp\left(\sum\limits_{k=0}^{\infty}s_{k}^{(j)}ch_{k}(V^{(j)})\right)
\]
where $s_{0}^{(j)},s_{1}^{(j)},s_{2}^{(j)}$,\ldots are formal indeterminates. 
We consider the special case where $V=N_{\sigma}\mathcal{P}$, 
which is the direct sum of line bundles $U_{j}(\sigma)$, for $j\in \sigma$, over $\mathcal{P}_{\sigma}$. 
For $j\in\sigma$, we set 
\[
s^{(j)}_{k}=\left\{
\begin{array}{lr}
-\text{log}U_{j}(\sigma),& k=0\\
(-1)^{k}(k-1)!U_{j}(\sigma)^{-k},& k\geq 1
\end{array}
\right.
\]
Then, we obtain the $(N_{\sigma}\mathcal{P},e_{\mathbb{T}^{-1}})$-twisted Gromov-Witten theory of 
$\mathcal{P}_{\sigma}$. Recall that $\mathcal{L}^{tw}$ is the Lagrangian cone of 
the $(N_{\sigma}\mathcal{P},e_{\mathbb{T}^{-1}})$-twisted Gromov-Witten theory of 
$\mathcal{P}_{\sigma}$. By direct computation, we obtain the following equation:
\begin{equation}\label{exp-s-j-k}
\exp(s^{(j)}(x))
=(U_{j}(\sigma)+x)^{-1},
\end{equation}
where $s^{(j)}(x):=\left(\sum\limits_{k=0}^{\infty}s_{k}^{(j)}\frac{x^{k}}{k!} \right)$.

As in \cite{CCIT09}, we introduce the function:
\[
G^{(j)}_{y}(x,z):=
\sum\limits_{l,m\geq 0}
s^{(j)}_{l+m-1}
\frac{B_{m}(y)}{m!}
\frac{x^{l}}{l!}
z^{m-1}\in \mathbb{C}[y,x,z,z^{-1}]
[[s_{0}^{(j)},s_{1}^{(j)},s_{2}^{(j)},\ldots]],
\]
By \cite{CCIT09}, the function $G^{(j)}_{y}(x,z)$ satisfies the following two relations:
\begin{align}\label{relation-G-j-1}
G_{y}^{(j)}(x,z)=& G_{0}^{(j)}(x+yz,z);\\
\label{relation-G-j-2}
G_{0}^{(j)}(x+z,z)=& G_{0}^{(j)}(x,z)+s^{(j)}(x).
\end{align}

Let $\theta_{j}=
\left(\sum_{i\not\in \sigma}
c_{ij}q_{i}(\partial/\partial q_{i})\right)$, where rational numbers $c_{ij}$ for $i\not\in \sigma$ and $j\in\sigma$ are defined by
\[
\bar{\rho}_i=\sum\limits_{i\in\sigma}c_{ij}\bar{\rho}_j, \text{ for } 1\leq i\leq n; \quad \bar{s}_{i}=\sum\limits_{j\in \sigma}c_{ij}\bar{\rho}_j, \text{ for } 1\leq i \leq m.
\]
Hence, rational numbers $c_{ij}$ satisfy the following equation:
\[
\lambda_j=-\sum\limits_{i\not\in\sigma}c_{ij}\lambda_i, \text{ for }\lambda\in \Lambda^S_\sigma\text{ and } j\in\sigma.
\]
We apply the differential operator 
$\exp(-\sum_{j\in\sigma}G_{0}^{(j)}(z\theta_{j},z))$ 
to (\ref{J-function-gerbe}), then we have:
\begin{align*}
 \textbf{L}:=
\exp\left(-\sum_{j\in\sigma}G_{0}^{(j)}(z\theta_{j},z)\right)
 \sum\limits_{\mathfrak D\in \NE(B)}
\sum\limits_{\substack{\lambda\in \Lambda^{S}_{\sigma}\\ \lambda_{i}-c_{1}(L_{i})\cdot\mathfrak D\in\mathbb{Z}_{\geq 0} \text{ if } i \not\in\sigma}}
Q^\mathfrak D
J_{\mathfrak D}(z,\tau)\\
 \times 
\left(
 \prod\limits_{i \not\in \sigma}
\frac{q_{i}^{\lambda_{i}-c_{1}(L_{i})\cdot\mathfrak D}\prod_{\langle a\rangle =0,a\leq 0}(az)}
{\prod_{\langle a\rangle =0,a\leq \lambda_{i}-c_{1}(L_{i})\cdot\mathfrak D}(az)}
\right)
1_{\sum_{i\not\in \sigma}\lambda_{i}b^{i}_\sigma} \\
\end{align*}
\begin{align*}
& =  \sum\limits_{\mathfrak D\in \NE(B)}
\sum\limits_{\substack{\lambda\in \Lambda^{S}_{\sigma}\\ \lambda_{i}-c_{1}(L_{i})\cdot\mathfrak D\in\mathbb{Z}_{\geq 0} \text{ if } i \not\in\sigma}}
Q^\mathfrak D
J_{\mathfrak D}(z,\tau)
\left(
 \prod\limits_{i \not\in \sigma}
\frac{q_{i}^{\lambda_{i}-c_{1}(L_{i})\cdot\mathfrak D}\prod_{\langle a\rangle =0,a\leq 0}(az)}
{\prod_{\langle a\rangle =0,a\leq \lambda_{i}-c_{1}(L_{i})\cdot\mathfrak D}(az)}
\right) \\ 
  & \times \exp\left(-\sum_{j\in\sigma}G_{0}^{(j)}\left(-z(\lambda_{j}-c_{1}(L_{i})\cdot\mathfrak D),z\right)\right)1_{\sum_{i\not\in \sigma}\lambda_{i}b^{i}_\sigma} 
\end{align*}
lies on the untwisted Lagrangian cone $\mathcal{L}$ 
of $\mathcal{P}_{\sigma}$. 
On the other hand, 
Tseng's orbifold quantum Riemann-Roch operator 
for $\oplus_{j\in\sigma}U_{j}(\sigma)$ is of the form:
\[
\Delta_{tw}:=
\bigoplus\limits_{b\in \bbox(\sigma)}
\exp\left(\sum\limits_{j\in\sigma}G_{b_{j}}^{(j)}(0,z)\right)
\]
This operator $\Delta_{tw}$ maps the untwisted Lagrangian cone $\mathcal{L}$ 
to the twisted Lagrangian cone $\mathcal{L}^{tw}$. Therefore

\begin{align*}
 \Delta_{tw}\textbf{L}=
& \sum\limits_{\mathfrak D\in \NE(B)}
\sum\limits_{\substack{\lambda\in \Lambda^{S}_{\sigma}\\ \lambda_{i}\in\mathbb{Z} \text{ if } i \not\in\sigma}}
Q^\mathfrak D J_{\mathfrak D}(z,\tau)
\left(
 \prod\limits_{i \not\in \sigma}
\frac{q_{i}^{\lambda_{i}-c_{1}(L_{i})\cdot\mathfrak D}\prod_{\langle a\rangle =0,a\leq 0}(az)}
{\prod_{\langle a\rangle =0,a\leq \lambda_{i}-c_{1}(L_{i})\cdot\mathfrak D}(az)}
\right)
\\
&\times
\exp
\left(
-\sum_{j\in\sigma}\left(G_{b_{j}}^{(j)}(0,z)-G_{0}^{(j)}(-z(\lambda_{j}-c_{1}(L_{j})\cdot\mathfrak D),z)\right)
\right)
1_{\sum_{i\not\in \sigma}\lambda_{i}b^{i}_\sigma}\\
= & \sum\limits_{\mathfrak D\in \NE(B)}
\sum\limits_{\substack{\lambda\in \Lambda^{S}_{\sigma}\\ \lambda_{i}\in\mathbb{Z} \text{ if } i \not\in\sigma}}
Q^\mathfrak D J_{\mathfrak D}(z,\tau)
\left(
 \prod\limits_{i \not\in \sigma}
\frac{q_{i}^{\lambda_{i}-c_{1}(L_{i})\cdot\mathfrak D}\prod_{\langle a\rangle =0,a\leq 0}(az)}
{\prod_{\langle a\rangle =0,a\leq \lambda_{i}-c_{1}(L_{i})\cdot\mathfrak D}(az)}
\right)
\\
&\times
\exp
\left(
-\sum_{j\in\sigma}\left(G_{0}^{(j)}(\langle -\lambda_{j} \rangle z,z)-G_{0}^{(j)}(-z(\lambda_{j}-c_{1}(L_{j})\cdot\mathfrak D),z)\right)
\right)
1_{\sum_{i\not\in \sigma}\lambda_{i}b^{i}_\sigma}\\
=& \sum\limits_{\mathfrak D\in \NE(B)}
\sum\limits_{\substack{\lambda\in \Lambda^{S}_{\sigma}\\ \lambda_{i}\in\mathbb{Z} \text{ if } i \not\in\sigma}}
Q^\mathfrak D J_{\mathfrak D}(z,\tau)
\left(
 \prod\limits_{i \not\in \sigma}
\frac{q_{i}^{\lambda_{i}-c_{1}(L_{i})\cdot\mathfrak D}\prod_{\langle a\rangle =0,a\leq 0}(az)}
{\prod_{\langle a\rangle =0,a\leq \lambda_{i}-c_{1}(L_{i})\cdot\mathfrak D}(az)}
\right)
\\
&\times
\left(
\prod\limits_{i \in \sigma}
\frac{\prod_{\langle a \rangle =\langle \lambda_{i}-c_{1}(L_{i})\cdot\mathfrak D\rangle,a\leq 0}\exp(-s^{(j)}(-az))}
{\prod_{\langle a\rangle =\langle \lambda_{i}-c_{1}(L_{i})\cdot\mathfrak D\rangle,a\leq \lambda_{i}-c_{1}(L_{i})\cdot\mathfrak D}\exp(-s^{j}(-az))}
\right)
1_{\sum_{i\not\in \sigma}\lambda_{i}b^{i}_\sigma}\\
=& \sum\limits_{\mathfrak D\in \NE(B)}
\sum\limits_{\substack{\lambda\in \Lambda^{S}_{\sigma}\\ \lambda_{i}\in\mathbb{Z} \text{ if } i \not\in\sigma}}
Q^\mathfrak D J_{\mathfrak D}(z,\tau)
\left(
 \prod\limits_{i \not\in \sigma}
\frac{q_{i}^{\lambda_{i}-c_{1}(L_{i})\cdot\mathfrak D}\prod_{\langle a\rangle =0,a\leq 0}(az)}
{\prod_{\langle a\rangle =0,a\leq \lambda_{i}-c_{1}(L_{i})\cdot\mathfrak D}(az)}
\right)
\\
&\times
\left( 
\prod\limits_{i \in \sigma}
\frac{\prod_{\langle a \rangle =\langle \lambda_{i}-c_{1}(L_{i})\cdot\mathfrak D\rangle,a\leq 0}(U_{i}(\sigma)+az)}
{\prod_{\langle a\rangle =\langle \lambda_{i}-c_{1}(L_{i})\cdot\mathfrak D\rangle,a\leq \lambda_{i}-c_{1}(L_{i})\cdot\mathfrak D}(U_{i}(\sigma)+az)}
\right)
1_{\sum_{i\not\in \sigma}\lambda_{i}b^{i}_\sigma}
= I^{S}_{\sigma} e^{-\sum^{n}_{i=1}U_{i}(\sigma)t_{i}/z}
\end{align*}
lies on $\mathcal{L}^{tw}$, where the second equation follows from (\ref{relation-G-j-1}), the third equation follows from (\ref{relation-G-j-2}) and fourth equation follows from (\ref{exp-s-j-k}).
This completes the proof of Theorem \ref{main-theorem}.


\begin{thebibliography}{30}	

\bibitem{AGV} D. Abramovich, T. Graber, A. Vistoli,
\emph{Algebraic orbifold quantum products,}
Orbifolds in mathematics and physics (Madison, WI, 2001), 1-24, Contemp. Math., 310, Amer. math. Soc., Providence, RI, 2002.


\bibitem{AJT10} E. Andreini, Y. Jiang, H.-H. Tseng, \emph{Gromov-Witten theory of banded gerbes over schemes}, arXiv:1101.5996.

\bibitem{AJT09}E. Andreini, Y. Jiang, H.-H. Tseng,
\emph{Gromov-Witten theory of root gerbes I: structure of genus $0$ moduli spaces,}
J. Differential Geom. 99 (2015), no. 1, 1--45.

\bibitem{BCS} L. Borisov, L. Chen, G. Smith,
\emph{The orbifold Chow ring of toric Deligne-Mumford stacks,} 
J. Amer. Math. Soc. 18 (2005), no.1, 193--215.

\bibitem{Brown} J. Brown,
\emph{Gromov-Witten Invariants of Toric Fibrations},
Int. Math. Res. Not. IMRN 2014, no. 19, 5437--5482.

\bibitem{CGT}
T. Coates, A. Givental, H.-H. Tseng,
\emph{Virasoro Constraints for Toric Bundles},
arXiv:1508.06282.

\bibitem{CCK} D. Cheong, I. Ciocan-Fontanine, B. Kim,
\emph{Orbifold Quasimap Theory},
Math. Ann. 363 (2015), no. 3-4, 777--816. 

\bibitem{CCIT} T. Coates, A. Corti, H. Iritani, H.-H. Tseng,
\emph{A Mirror Theorem for Toric Stacks},
Compos. Math. 151 (2015), no. 10, 1878--1912. 

\bibitem{CCIT09}T. Coates, A. Corti, H. Iritani, H.-H. Tseng,
\emph{Computing Genus-Zero Twisted Gromov-Witten Invariants,}
Duke Math. J., 147 (2009), no. 3,377--438. 

\bibitem{Giv1} A. Givental,
\emph{Gromov-Witten invariants and quantization of quadratic Hamiltonians}, 
Mosc. Math. J. 1 (2001), no. 4, 551--568.

\bibitem{Giv2} A. Givental, 
\emph{Symplectic geometry of Frobenius structures}, 
In: ``Frobenius manifolds'', 91--112, Aspects Math., E36, Friedr. Vieweg, Wiesbaden, 2004. 

\bibitem{Iritani} H. Iritani, 
\emph{An integral structure in quantum cohomology and mirror symmetry for toric orbifolds,} 
Adv. Math. 222 (2009), 1016--1079.

\bibitem{Jiang} Y. Jiang,
\emph{The orbifold cohomology ring of simplicial toric stack bundles,} 
Illinois J. Math. 52 (2008), no. 2, 493--514.

\bibitem{LLW} Y.-P. Lee, H.-W. Lin, C.-L. Wang, \emph{Invariance of quantum rings under ordinary flops:I}, arXiv:1109.5540.

\bibitem{Liu} C.-C. M. Liu,
\emph{Localization in Gromov-Witten theory and Orbifold Gromov-Witten Theory,}
In: ``Handbook of Moduli'', Volume II, 353--425, Adv. Lect. Math., (ALM) 25, International Press and Higher Education Press, 2013.

\bibitem{SU} P. Sankaran, V. Uma, 
\emph{Cohomology of toric bundles},
Comment. Math. Helv. 78 (2003), no. 3, 540--554. 


\bibitem{Tel} C. Teleman,
\emph{The structure of 2D semi-simple field theories}, 
Invent. Math. 188 (2012), no. 3, 525--588.

\bibitem{Tseng} H.-H. Tseng, 
\emph{Orbifold quantum Riemann-Roch, Lefschetz, and Serre,} 
Geom. Topol. 14 (2010), no. 1, 1--81.








\end{thebibliography}
\end{document}